\documentclass[11pt,a4paper]{amsart}

\usepackage{amsmath,amscd}
\usepackage{amssymb}
\usepackage{amsthm}

\usepackage{graphicx}

\usepackage{mathrsfs}
\usepackage{comment}
\usepackage[ocgcolorlinks, linkcolor=blue]{hyperref}

\usepackage{calc}
             {\begin{list}{\arabic{enumi}.}{\usecounter{enumi}%
              \setlength{\labelsep}{0.5em}%
              \settowidth{\labelwidth}{(\arabic{enumi})}%
              \setlength{\leftmargin}{\labelwidth+\labelsep}}}%
             {\end{list}}

\newtheorem{Theorem}{Theorem}[section]

\newtheorem{Lemma}[Theorem]{Lemma}
\newtheorem{Corollary}[Theorem]{Corollary}

\newtheorem{Proposition}[Theorem]{Proposition}

\theoremstyle{definition}
\newtheorem{Definition}[Theorem]{Definition}
\newtheorem*{DefinitionNoNumber}{Definition}

\newtheorem{Remark}[Theorem]{Remark}

\numberwithin{equation}{section}

\newcommand{\beqn}{\begin{equation}}
\newcommand{\eeqn}{\end{equation}}

\def\ep{{\epsilon}}
\def\nn{\nonumber}
\def\pa{\partial}

\newcommand{\Zl}{Z_{\lambda}}
\newcommand{\Omb}{{\overline{\Omega}}}

\def\cleq{\lesssim}
\def\cgeq{\gtrsim}
\newcommand {\R} {\mathbb{R}}

\newcommand {\p} {\partial}

\newcommand{\tre}{\textcolor{red}}

\def\vt{{\tilde{v}}}

\def\ft{{\tilde{f}}}

\def\pt{{\tilde{p}}}
\def\qt{{\tilde{q}}}

\def\rhob{{\bar{\rho}}}
\def\pb{{\bar{p}}}

\def\Bb{{\bar{B}}}

\newcommand{\mR}{\mathbb{R}}                    % Formatting for R
\newcommand{\mC}{\mathbb{C}}                    % Formatting for C
\newcommand{\abs}[1]{\lvert #1 \rvert}          % Formatting for the absolute value
\newcommand{\norm}[1]{\lVert #1 \rVert}         % Formatting for the norm
\newcommand{\ol}[1]{\overline{#1}}

\newcommand{\im}{\mathrm{Im}}

\newcommand{\ehat}{\,\widehat{\rule{0pt}{6pt}}\,}

\newcommand{\etilde}{\,\tilde{\rule{0pt}{6pt}}\,}
\newcommand{\ebreve}{\,\breve{\rule{0pt}{6pt}}\,}

\newcommand{\eps}{\varepsilon}

\newcommand{\supp}{\mathrm{supp}}

\newcommand{\even}{\mathrm{even}}
\newcommand{\odd}{\mathrm{odd}}

\newcounter{sidenote}

\begin{document}

\title[Fixed angle inverse scattering]{Fixed angle inverse scattering for almost symmetric or controlled perturbations}

\author{Rakesh}
\address{Department of Mathematical Sciences, University of Delaware, Newark, DE 19716, USA}
\email{rakesh@udel.edu}

\author[M. Salo]{Mikko Salo}
\address{University of Jyvaskyla, Department of Mathematics and Statistics, PO Box 35, 40014 University of Jyvaskyla, Finland}
\email{mikko.j.salo@jyu.fi}

%\subjclass{53C25, 53C21, 58F17, 35J15}

%\date{\today}

%\maketitle

\begin{abstract}
We consider the fixed angle inverse scattering problem and show that a compactly supported potential is uniquely determined by its 
scattering amplitude for two opposite fixed angles. We also show that almost symmetric or horizontally controlled potentials are uniquely 
determined by their fixed angle scattering data. This is done by establishing an equivalence between the frequency domain and
the time domain formulations of the problem, and by solving the time domain problem by extending the methods of 
\cite{RakeshSalo} which adapts the ideas introduced in \cite{BukhgeimKlibanov} and \cite{IY01}
on the use of Carleman estimates for inverse problems.
\end{abstract}

\maketitle

\section{Introduction} \label{sec_introduction}

%In this article $n$ denotes a positive integer $n \geq 2$.
In inverse scattering problems the objective is to determine certain properties of a scatterer from measurements that are made far away. In stationary scattering theory, the measurements are often formulated in terms of the \emph{scattering amplitude}. If $\lambda > 0$ is a frequency and if $\omega \in S^{n-1}$, consider the plane wave $\psi^i(x) = e^{i\lambda \omega \cdot x}$ propagating in direction $\omega$. The interaction of this plane wave with a real valued scattering potential $q \in C^{\infty}_c(\mR^n)$ is described by the outgoing eigenfunction (or distorted plane wave) $\psi_q = \psi^i + \psi^s_q$, which solves the Schr\"odinger equation 
\[
(-\Delta + q - \lambda^2) \psi_q = 0 \text{ in $\mR^n$}
\]
and where the scattered wave $\psi^s_q$ is \emph{outgoing}. There are several equivalent ways to describe the outgoing condition (or Sommerfeld radiation condition), but for us it is enough that $\psi^s_q$ is given by the outgoing resolvent applied to the compactly supported function $-q \psi^i$:
\[
\psi^s_q = (-\Delta+q-(\lambda+i0)^2)^{-1}(-q \psi^i).
\]
Writing $x = r\theta$ where $r \geq 0$ and $\theta \in S^{n-1}$, the scattered wave has the asymptotics 
\[
\psi^s_q(r\theta) = e^{i\lambda r} r^{-\frac{n-1}{2}} a_q(\lambda, \theta, \omega) + o(r^{-\frac{n-1}{2}}) \qquad \text{as $r \to \infty$}.
\]
The function $a_q$ is called the \emph{scattering amplitude}, or \emph{far field pattern}, corresponding to the potential $q$. One could interpret $a_q(\lambda,\theta,\omega)$ as a scattering measurement for $q$ that corresponds to sending a plane wave at frequency $\lambda > 0$ propagating in the direction $\omega \in S^{n-1}$ and measuring the scattered wave in
the direction $\theta \in S^{n-1}$. See e.g.\ \cite{ColtonKress, DyatlovZworski, Melrose, Yafaev} for more details on these facts.

Next, we formulate four fundamental inverse scattering problems, related to recovering a potential from (partial) knowledge of its quantum mechanical scattering amplitude: \\
\begin{enumerate}
\item[1.] {\bf Full data.} Recover $q$ from $a_q$.
\item[2.] {\bf Fixed frequency.} Recover $q$ from $a_q(\lambda_0,\,\cdot\,,\,\cdot\,)$ with $\lambda_0 > 0$ fixed.
\item[3.] {\bf Backscattering.} Recover $q$ from $a_q(\lambda,\omega,-\omega)$ for $\lambda > 0$, $\omega \in S^{n-1}$.
\item[4.] {\bf Fixed angle.} Recover $q$ from $a_q(\,\cdot\,, \omega_0, \,\cdot\,)$ where $\omega_0 \in S^{n-1}$ is fixed. \\
\end{enumerate}
The full data problem is formally overdetermined when $n \geq 2$, since one seeks to recover a function of $n$ variables from a function of $2n-1$ variables. Similarly, the fixed frequency problem is formally overdetermined when $n \geq 3$ (it is formally determined when $n=2$). Both of these problems have been solved; we only mention that one can determine $q$ 
from the high frequency asymptotics of $a_q$ \cite{Saito} and that the fixed frequency problem is equivalent to a variant of the inverse conductivity problem of Calder\'on addressed in \cite{SylvesterUhlmann, Bukhgeim}. There have been many related works and we refer to \cite{Uhlmann_asterisque, Novikov_survey, Uhlmann2014} for references.

The backscattering and the fixed angle inverse scattering problems are formally determined in any dimension (both the unknown and the data depend on $n$ variables). The one-dimensional case is well understood \cite{Marchenko, DeiftTrubowitz}. Known results for $n \geq 2$ include uniqueness for potentials that are small or belong to a generic set \cite{EskinRalston, Stefanov_generic, MelroseUhlmann, BarceloEtAl}, recovery of main singularities \cite{GreenleafUhlmann, OlaPaivarintaSerov, Ruiz}, identification of the zero potential in fixed angle scattering \cite{BaylissLiMorawetz}, and recovery of angularly controlled potentials from backscattering data \cite{RakeshUhlmann}. See the references in \cite{RakeshUhlmann, Meronno_thesis} for further results. However, these problems remain open in general.

We establish several new results for the fixed angle inverse scattering problem, when $n \geq 2$. Our first result shows that a compactly supported potential is uniquely determined by the scattering amplitude at two opposite fixed angles.

\begin{Theorem} \label{thm_main1}
Fix $\omega \in S^{n-1}$, $n \geq 2$, and let $q_1, q_2 \in C^{\infty}_c(\mR^n)$ be real valued. If 
\begin{align*}
a_{q_1}(\lambda, \omega,\theta) = a_{q_2}(\lambda, \omega,\theta) \quad \text{ and } \quad a_{q_1}(\lambda, -\omega,\theta) = a_{q_2}(\lambda, -\omega,\theta)
\end{align*}
for all $\lambda > 0$ and $\theta \in S^{n-1}$, then $q_1 = q_2$.
\end{Theorem}

As a corollary, it follows that a reflection symmetric potential is uniquely determined by its fixed angle scattering data.

\begin{Corollary} \label{thm_main2}
Fix $\omega \in S^{n-1}$ and let $q_1, q_2 \in C^{\infty}_c(\mR^n)$ be reflection symmetric in the sense that 
\[
q_j(\eta + t\omega) = q_j(\eta - t\omega), \qquad \text{for all } \eta \in \R^n \text{ with }\eta \perp \omega, \ t \in \mR, \ j=1,2.
\]
If $a_{q_1}(\lambda,\omega,\theta) = a_{q_2}(\lambda,\omega,\theta)$ for all $\lambda > 0$, $\theta \in S^{n-1}$ then $q_1 = q_2$.
\end{Corollary}

We show that the above results follow directly from corresponding results for the time domain inverse problems that were studied in \cite{RakeshSalo}. In fact, in this paper we show that the time and frequency domain formulations of the fixed angle scattering problem are equivalent. When $n \geq 3$ is odd, such an equivalence has been discussed in \cite{Melrose, Uhlmann_backscattering, MelroseUhlmann_bookdraft} in the context of Lax-Phillips scattering theory. We give a direct argument that works in any dimension.

The work \cite{RakeshSalo} was concerned with wave equation inverse problems with two measurements. In this paper we also extend the methods of \cite{RakeshSalo} and obtain improved results in the case of a single measurement (i.e.\ fixed angle scattering).

The first improved result considers potentials that satisfy a generalized reflection symmetry or small perturbations of such potentials. We fix an $(n-1) \times (n-1)$ orthogonal matrix $A$, take $\omega = e_n$ and, for any $x \in \R^n$, write $x = (y,z)$ 
with $y \in \mR^{n-1}$ and $z \in \mR$. For any function $p$ on $\mR^n$, we define its generalized even and odd parts
as 
\begin{align}
p_{\even}(y,z) &:= \frac{1}{2} \left[ p(y,z) + p(Ay, -z) \right], \label{p_even_definition} \\
p_{\odd}(y,z) &:= \frac{1}{2} \left[ p(y,z) - p(Ay, -z) \right]. \label{p_odd_definition} 
\end{align}

\begin{Theorem} \label{thm_main3}
Let $M > 1$ and $\omega = e_n$. There is an $\eps = \eps(M) > 0$ with the following property: if $q, p \in C^{\infty}_c(\mR^n)$ are supported 
in $\ol{B}$ and $\norm{q}_{C^{n+4}} \leq M$, $\norm{p}_{C^{n+4}} \leq M$, then the condition
\[
a_{q+p}(\lambda,\omega,\theta) = a_{q}(\lambda,\omega,\theta) \qquad \text{for all } \lambda > 0 \text{ and } \theta \in S^{n-1}
\]
implies $p = 0$, provided
\[
\norm{p_{\odd}}_{H^1(B)} \leq \eps \norm{p}_{L^2(B)}
\]
or  
\[
\norm{p_{\even}}_{H^1(B)} \leq \eps \norm{p}_{L^2(B)}.
\]
\end{Theorem}

In particular, if $q \in C^{\infty}_c(\mR^n)$ satisfies a generalized reflection symmetry in the sense that $q_{\odd} = 0$ or $q_{\even} = 0$, then $q$ is uniquely determined by its fixed angle scattering data.

The next result involves functions which are horizontally controlled, as defined next.

\begin{DefinitionNoNumber}
Given $M, \eps \geq 0$, a function $r(y,z) \in H^1(\mR^n)$, with support in $\{ \abs{y} \leq 1 \}$, is said to be horizontally $(M,\eps)$-controlled if 
\[
\int_{\mR^{n-1}} \abs{\nabla_y r(y,z)}^2 \,dy \leq M \int_{\mR^{n-1}} \abs{r(y,z)}^2 \,dy + \eps \int_{\mR^{n-1}} \abs{\partial_z r(y,z)}^2 \,dy, 
\]
for almost every $z \in (-1,1)$.
\end{DefinitionNoNumber}

\begin{Theorem} \label{thm_main4}
Let $M > 1$ and $\omega = e_n$. There is an $\eps = \eps(M) > 0$ with the following property: if $q, p \in C^{\infty}_c(\mR^n)$ are supported in 
$\ol{B}$ and $\norm{q}_{C^{n+4}} \leq M$, $\norm{p}_{C^{n+4}} \leq M$, then the condition
\[
\text{$a_{q+p}(\lambda,\omega,\theta) = a_{q}(\lambda,\omega,\theta)$ \qquad for all $\lambda > 0$ and $\theta \in S^{n-1}$}
\]
implies $p = 0$, provided the function 
\[
r(y,z) := \int_{-\infty}^z p(y,s) \,ds, \qquad (y,z) \in \R^n,
\]
is horizontally $(M,\eps)$-controlled.
\end{Theorem}

For example, the fixed angle scattering data determines uniquely any perturbation $p(y,z)$ of the form 
\[
p(y,z) = \sum_{j=1}^N p_j(z) \varphi_j(y), \qquad (y,z) \in \R^n
\]
where $\varphi_1, \ldots, \varphi_N$ are fixed linearly independent functions in $C^{\infty}_c(\mR^{n-1})$ and $p_j$ are 
arbitrary functions in $C^{\infty}_c(\mR)$ supported in a fixed interval - see Lemma \ref{lemma_p_controlled}. Theorem \ref{thm_main4} is analogous to the result for angularly controlled potentials in backscattering \cite{RakeshUhlmann} or the result in \cite{RomanovAnalytic} for potentials which are analytic in $y$ (see also \cite{SacksSymes}).
 
We prove the above theorems by reducing them (see Section \ref{sec_equivalence_frequency_time}) to certain inverse problems for the wave equation in the time 
 domain. These time domain problems are solved by extending the methods of \cite{RakeshSalo} which adapted the ideas 
 introduced
 in \cite{BukhgeimKlibanov} and \cite{IY01} on the use of Carleman estimates for formally determined inverse problems.
Please refer to  \cite{Kh89, Bu00, Be04, Is06, Kl13, SU13, BY17} for further details about this method and its variants.

More specifically, our proofs will proceed as follows:
\begin{enumerate}
\item[1.]
First, the time domain fixed angle scattering problem is 
 reduced to an inverse source problem for the wave equation. If the source were zero, this would be a standard unique continuation problem which could be solved using a Carleman estimate. Here the source is nonzero but it has a special form: the unknown part of the source is time-independent and related to the trace of the solution on a certain characteristic boundary.
\item[2.]
We then invoke a Carleman estimate  for the wave equation with boundary terms
which estimates the solution in terms of the source and the boundary terms. Because of Step 1, the source can be estimated 
by the trace of the solution on the characteristic part of the boundary. If the Carleman weight is pseudoconvex and decays rapidly away from the characteristic boundary, then it just remains to control the characteristic boundary terms.
\item[3.]
If the Carleman weight has the properties in Step 2, then the characteristic boundary term will have an adverse sign. We deal with the adverse sign term either by using a reflection argument, leading to Theorem \ref{thm_main3}, or by assuming that 
 the adverse sign term is controlled by other boundary terms, leading to Theorem \ref{thm_main4}.
\end{enumerate}

We emphasize that this method leads to uniqueness and Lipschitz stability results for the time domain inverse problems - see Theorems \ref{thm_single_measurement_perturbation1} and \ref{thm_single_measurement_perturbation2} for precise statements. Uniqueness in the frequency domain fixed angle problem then follows from the reduction in Section \ref{sec_equivalence_frequency_time} (stability does not follow immediately, since the reduction involves analytic continuation). In our earlier work \cite{RakeshSalo}, an extension argument and a Carleman estimate in the extended domain were used for proving an analogue of Theorem \ref{thm_single_measurement_perturbation1}. A similar extension argument could be used to prove 
Theorem \ref{thm_single_measurement_perturbation1}. However, in this paper, instead we use a Carleman estimate with explicit boundary terms, which turns out to be simpler and contains more information than the extension method. This new method also makes it possible to prove Theorem \ref{thm_main4}.

This work is organized as follows. Section \ref{sec_introduction} is the introduction, Section \ref{sec_time_domain_setting} introduces the time domain setting for the fixed angle scattering problem and contains some useful facts from \cite{RakeshSalo} and Sections \ref{sec_almost_reflection_symmetric} and \ref{sec_horizontally_controlled} contain the proofs of Theorems \ref{thm_single_measurement_perturbation1} and \ref{thm_single_measurement_perturbation2} respectively. In Section \ref{sec_equivalence_frequency_time}, we prove the equivalence of time and frequency domain scattering measurements 
which leads to Theorems \ref{thm_main1} to \ref{thm_main4}. Finally, Appendix \ref{sec_carleman_estimate} contains the derivation of a Carleman estimate with boundary terms for the wave equation with a pseudoconvex weight. This is well known except for the explicit form of the boundary terms, which is needed in our proofs; hence we give a detailed argument.

\subsection*{Acknowledgements}

Rakesh's work was supported by funds from an NSF grant DMS 1615616. M Salo's work was supported by the Academy of Finland (grants 284715 and 309963) and by the European Research Council under Horizon 2020 (ERC CoG 770924).

\section{The time domain setting} \label{sec_time_domain_setting}

In this section we recall, from \cite{RakeshSalo}, some notation and basic facts for the time domain inverse problem.
The open unit ball in $\R^n$ is denoted by $B$ and $S$ is its boundary, $ \Box = \pa_t^2 - \Delta_x$ is the wave operator and 
$q(x)$ is a smooth function on $\R^n$ with support in $\ol{B}$.
The vector $e_n=(0,0, \cdots,1)$, parallel to the $z$-axis, is the
{\bf fixed} direction of the incoming plane wave and given $x \in \R^n$, we write $x=(y,z)$ with $y \in \R^{n-1}, z \in \R$.

Let $U_q(x,t) = U_q(x,t,e_n)$ be the solution of the IVP 
\[
(\Box + q)U_q = 0 \text{ in $\mR^{n+1}$}, \qquad U_q|_{\{t<-1\}} = \delta(t-z).
\]
We can express $U_q$ in the form $U_q(x,t) = \delta(t-z) + u_q(x,t)$ where $u_q(x,t) = u_q(x,t,e_n)$ is the unique 
solution of the IVP 
\begin{equation} \label{uq_forcing_problem}
(\Box + q)u_q = -q(x) \delta(t-z) \text{ in $\mR^{n+1}$}, \qquad u_q|_{\{t < -1\}} = 0.
\end{equation}
This solution has the following properties (see \cite[Proposition 1.1]{RakeshSalo}).

\begin{Proposition} \label{prop_uq_basic}
There is a unique distributional solution $u_q$ of \eqref{uq_forcing_problem}. The distribution $u_q(x,t)$ is supported in 
$\{ t \geq z \}$
and has a unique representation as a smooth function on $\{t \geq z\}$ which is also the unique smooth solution of 
 the characteristic initial value problem 
\begin{align*}
(\Box + q) u_q &= 0 \text{ in $\{ t > z \}$}, \\
u_q(y,z,z) &= -\frac{1}{2} \int_{-\infty}^z q(y,s) \,ds \quad \text{for all } ~ (y,z) \in \R^n, \\
u_q(x,t) &= 0 \text{ in $\{ z < t < -1 \}$}.
\end{align*}
For any $M>0, T > 1$ there is a $C = C(M,T) > 0$ such that if $\norm{q}_{C^{n+4}} \leq M$ then
\[
\norm{u_q}_{L^{\infty}(\{ z \leq t \leq T \} )} \leq C.
\]
\end{Proposition}
Below, we regard the distribution $u_q(x,t)$ as a function on $\R^{n+1}$ which is zero on $\{ t < z\}$ and
is a smooth function on $\{ t \geq z \}$.

The single measurement inverse problem can be stated as follows:
\[
\text{Given $u_q|_{S \times (-1,T)}$ for some $T$, determine $q$ in $\mR^n$.}
\]
This corresponds to determining an inhomogeneity $q$ living inside $B$ by sending a plane wave $\delta(t-z)$ and measuring the scattered wave $u_q$ on the boundary of $B$ until time $T$.

We reduce this inverse problem to a unique continuation problem for the wave equation. To this end define the following subsets of $\mR^n \times \mR$: 
\begin{align*}
Q := B \times (-T,T), & \qquad \Sigma := S \times (-T,T), 
\\
Q_{\pm} := Q \cap \{ \pm(t-z) > 0 \},  & \qquad \Sigma_{\pm} := \Sigma \cap \{ \pm (t-z) \geq 0 \}, \\
\Gamma := \ol{Q} \cap \{t=z\}, & \qquad \Gamma_{\pm T} := \ol{Q} \cap \{t=\pm T\}.
\end{align*}
We will also need the vector fields 
\[
Z := \frac{1}{\sqrt{2}}(\partial_t + \partial_z), \qquad N := \frac{1}{\sqrt{2}}(\partial_t - \partial_z);
\]
note that $Z$ is tangential to $\Gamma$ and $N$ is normal to $\Gamma$.

Next, we state a result about a specific Carleman weight for the wave operator, which follows
from the discussion in \cite[Section 2.3]{RakeshSalo} and \cite[Lemma 3.2]{RakeshSalo}
(see Appendix \ref{sec_carleman_estimate} for the definition of a strongly pseudoconvex function). Note that the roles
of $\phi$ and $\psi$ in this paper are the reverse of the roles they play in \cite{RakeshSalo}.

\begin{Lemma} \label{lemma_pseudoconvex_weight}
Define 
\[
\psi(y,z,t) := 5(a-z)^2 + 5 \abs{y}^2 - (t-z)^2, \qquad (y,z) \in \R^n, ~ t \in \R,
\]
Given $T > 6$, there exists $a > 1$ such that 
\begin{itemize}
\item  the function $\phi = e^{\lambda \psi}$ is 
strongly pseudoconvex w.r.t $\Box$ in a (fixed) neighborhood of $\ol{Q}$ for sufficiently large $\lambda > 0$,
\item the smallest value of $\phi$ on $\Gamma$ is strictly larger than the largest value of $\phi$ on 
$\Gamma_T \cup \Gamma_{-T}$, 
\item the function 
\[
h(\sigma) := \sup_{(y,z) \in \ol{B}} \int_{-T}^T e^{2\sigma(\phi(y,z,t)-\phi(y,z,z))} \,dt
\]
satisfies $\lim_{\sigma \to \infty} h(\sigma) = 0$.
\end{itemize}
\end{Lemma}

For use later, we also quote the following energy estimates from \cite[Lemmas 3.3--3.5]{RakeshSalo}.

\begin{Lemma} \label{lemma_energy_estimate_exterior}
Let $T > 1$ and $p \in C^{\infty}_c(\mR^n)$ be supported in $\ol{B}$. If $\alpha(x,t)$ is a smooth function on
$\{ t \geq z \}$ satisfying 
\begin{align*}
\Box \alpha &= 0 \text{ in $\{ (x,t) \,;\, \abs{x} > 1 \text{ and } t > z \}$}, \\
\alpha(y,z,z) &= \int_{-\infty}^z p(y,s) \,ds \text{ on $\{ \abs{x} > 1$\}}, \\
\alpha &= 0 \text{ in $\{ z < t < -1 \}$},
\end{align*}
then 
\[
\norm{\p_{\nu} \alpha}_{L^2(\Sigma_+)} \lesssim \norm{\alpha}_{H^1(\Sigma_+)} + \norm{\alpha}_{H^1(\Sigma_+ \cap \Gamma)}
\]
with the constant dependent only on $T$.
\end{Lemma}

\begin{Lemma} \label{lemma_energy_estimate_top}
Let $T > 1$ and $q \in C^{\infty}_c(\mR^n)$ be supported in $\ol{B}$. For every $\alpha \in C^{\infty}(\ol{Q}_+)$ we have 
\begin{equation*}
\norm{\alpha}_{L^2(\Gamma_T)} + \norm{\nabla_{x,t} \alpha}_{L^2(\Gamma_T)} \lesssim \norm{\alpha}_{H^1(\Gamma)} + \norm{(\Box+q)\alpha}_{L^2(Q_+)} + \norm{\alpha}_{H^1(\Sigma_+)} + \norm{\p_{\nu} \alpha}_{L^2(\Sigma_+)}
\end{equation*}
with the constant dependent only on $\norm{q}_{L^{\infty}}$ and $T$.
\end{Lemma}

\begin{Lemma} \label{lemma_energy_estimate_bottom}
Let $T > 1$, $q \in C^{\infty}_c(\mR^n)$ be supported in $\ol{B}$ and $\phi \in C^2(\ol{Q}_+)$. 
There are constants $C, \sigma_0 > 1$, depending only on $\norm{q}_{L^{\infty}}$, $\norm{\phi}_{C^2(\ol{Q}_+)}$ and $T$,
such that for every $\alpha \in C^{\infty}(\ol{Q}_+)$ and $\sigma \geq \sigma_0$ one has the estimate 
\begin{multline*}
\sigma^2 \norm{e^{\sigma \phi} \alpha}_{L^2(\Gamma)}^2  + \norm{e^{\sigma \phi} \nabla_{\Gamma} \alpha}_{L^2(\Gamma)}^2 \leq C \Big[ \sigma^3 \norm{e^{\sigma \phi} \alpha}_{L^2(Q_+)}^2 + \sigma \norm{e^{\sigma \phi} \nabla_{x,t} \alpha}_{L^2(Q_+)}^2 \\
 + \norm{e^{\sigma \phi} (\Box+q)\alpha}_{L^2(Q_+)}^2 + \sigma^2 \norm{e^{\sigma \phi} \alpha}_{L^2(\Sigma_+)}^2 + \norm{e^{\sigma \phi} \nabla_{x,t} \alpha}_{L^2(\Sigma_+)}^2 \Big].
\end{multline*}
\end{Lemma}

\section{Almost reflection symmetric perturbations} \label{sec_almost_reflection_symmetric}

We will use the notation from Section \ref{sec_time_domain_setting}. If $A$ is an $(n-1) \times (n-1)$ orthogonal matrix and 
$\sigma \in \{ +1, -1 \}$, we define 
\[
\breve{p}(y,z) := \frac{1}{2} \left[ p(y,z) - \sigma p(Ay, -z) \right].
\]
Comparing with \eqref{p_even_definition}--\eqref{p_odd_definition}, one has $\breve{p} = p_{\odd}$ when $\sigma=1$ and $\breve{p} = p_{\even}$ when $\sigma=-1$. The following result solves the time domain analogue of the fixed angle scattering problem for almost reflection symmetric potentials and gives a Lipschitz stability estimate.

\begin{Theorem} \label{thm_single_measurement_perturbation1}
Let $M > 1$, $T > 6$ and $\sigma \in \{1,-1\}$. There exist positive constants $C$ and $\eps$, depending only on $M$ and 
$T$, with the following property: if 
$q, p \in C^{\infty}_c(\mR^n)$ are supported in $\ol{B}$ and $\norm{q}_{C^{n+4}} \leq M$, $\norm{p}_{C^{n+4}} \leq M$, then 
\[
\norm{p}_{L^2(B)} \leq C ( \norm{u_{q+p}-u_q}_{H^1(\Sigma_+)} + \norm{u_{q+p}-u_q}_{H^1(\Sigma_+ \cap \Gamma)} )
\]
provided
\[
\norm{\breve{p}}_{H^1(B)} \leq \eps \norm{p}_{L^2(B)}.
\]
\end{Theorem}

Theorem \ref{thm_single_measurement_perturbation1} will follow from the next result which proves uniqueness and stability for a certain linear inverse problem.

\begin{Proposition} \label{prop_single_measurement_perturbation}
Let $M > 1$ and $T > 6$. There is a $C(M,T) > 0$ so that if
\begin{align*}
(\Box + q_{\pm}) w_{\pm}(x,t) = (Zw_{\pm})(x,z) \, f_{\pm}(x,t) \, \text{ in $Q_{\pm}$},
\end{align*}
for some $q_{\pm} \in C^{\infty}_c(\mR^n)$ supported in $\ol{B}$, $f_{\pm} \in L^{\infty}(Q_{\pm})$ and
 $w_{\pm} \in H^2(Q_{\pm})$  with $\norm{q_{\pm}}_{L^{\infty}(B)} \leq M$, 
 $\norm{f_{\pm}}_{L^{\infty}(Q_{\pm})} \leq M$, then 
\begin{equation*}
\sum_{\pm} \norm{w_{\pm}}_{H^1(\Gamma)} \leq C \Big[ \norm{w_+-w_-}_{H^1(\Gamma)} + \sum_{\pm} (\norm{w_{\pm}}_{H^1(\Sigma_{\pm})} + \norm{\p_{\nu} w_{\pm}}_{L^2(\Sigma_{\pm})}) \Big].
\end{equation*}
\end{Proposition}
Note the special structure of the right hand side of the PDE. It has the $(Zw_{\pm})(x,z)$ term which resides on $\Gamma$ and hence
the appropriate Carleman weight helps us absorb the right hand side of the PDE into the left hand side of the inequality. That is why there is no
$f_\pm$ term on the right hand side of the estimate.

\begin{proof}[Proof of Theorem \ref{thm_single_measurement_perturbation1}]
Assume that $q, p$ and $\sigma$ are as in the statement of the theorem and define 
\[
w := u_{q+p} - u_q,
\]
where $u_{q+p}$ and $u_q$ are as in Proposition \ref{prop_uq_basic}. The function $w$ is smooth on the region
$t \geq z$, solves the equation 
\[
(\Box + q)w = -p(x) u_{q+p} \text{ in $Q_+$}
\]
and on $\Gamma$, the bottom part of the boundary of $Q_+$, has the trace
\beqn
w(y,z,z) = -\frac{1}{2} \int_{-\infty}^z p(y,s) \,ds, \qquad \text{for all } ~ (y,z) \in \Bb,
\label{eq:wzzp}
\eeqn
so 
$Zw(y,z,z) = -\frac{1}{2\sqrt{2}} p(y,z)$. Thus, taking 
\[
w_+ = w, \qquad q_+ = q, \qquad f_+ = 2 \sqrt{2} u_{q+p},
\]
one has $(\Box+q_+)w_+ = (Zw_+)|_{\Gamma} f_+$ in $Q_+$. Moreover, $\norm{f_+}_{L^{\infty}(Q_+)} \leq C(M,T)$ by Proposition \ref{prop_uq_basic}.

Next, define $w_-$ in $Q_-$ by reflection, that is
\[
w_-(y,z,t) = -\sigma w_+(Ay,-z,-t), \qquad (y,z,t) \in Q_-;
\]
then on $Q_-$ we have
\begin{align*}
\Box w_-(y,z,t) &= -\sigma (\Box w_+)(Ay,-z,-t) \\
 &= -\sigma (-q_+ w_+ + (Zw_+)|_{\Gamma} f_+)(Ay,-z,-t).
\end{align*}
Further, a tangential derivative of the trace of $w_-$ on $\Gamma$ is given by
\[
Z w_-(y,z,z) = \sigma (Zw_+)(Ay,-z,-z), \qquad (y,z) \in \Bb,
\]
so, if we define
\[
q_-(y,z) = -\sigma q_+(Ay,-z,-t), \qquad f_-(y,z) = -f_+(Ay,-z,-t), \qquad (y,z) \in \Bb,
\]
then $(\Box+q_-)w_- = (Zw_-)|_{\Gamma} f_-$ in $Q_-$ and $\norm{f_-}_{L^{\infty}(Q_-)} \leq C(M,T)$.

Thus, we are exactly in the situation of Proposition \ref{prop_single_measurement_perturbation}, which implies that 
\begin{multline*}
\sum_{\pm} \norm{w_{\pm}}_{H^1(\Gamma)} \leq C(M,T) ( \norm{w_+-w_-}_{H^1(\Gamma)} \\
 + \sum_{\pm} (\norm{w_{\pm}}_{H^1(\Sigma_{\pm})} + \norm{\p_{\nu} w_{\pm}}_{L^2(\Sigma_{\pm})} + \norm{w_{\pm}}_{H^1(\Sigma_{\pm} \cap \Gamma)}) ).
\end{multline*}
By Lemma \ref{lemma_energy_estimate_exterior}, which applies in $Q_+$ as well as in $Q_-$, one has 
\[
\norm{\p_{\nu} w_{\pm}}_{L^2(\Sigma_{\pm})} \leq C(T) (\norm{w_{\pm}}_{H^1(\Sigma_{\pm})} + \norm{w_{\pm}}_{H^1(\Sigma_{\pm} \cap \Gamma)}).
\]
Using the definition of $w_-$, one also has 
\[
\norm{w_{-}}_{H^1(\Sigma_{-})} + \norm{w_{-}}_{H^1(\Sigma_{-} \cap \Gamma)} \leq \norm{w_{+}}_{H^1(\Sigma_{+})} + \norm{w_{+}}_{H^1(\Sigma_{+} \cap \Gamma)}.
\]
Moreover, using (\ref{eq:wzzp}) and the definition of $w_+$, $Z$, we have 
\[
\norm{p}_{L^2(B)} \lesssim \norm{Zw_+}_{L^2(\Gamma)} \leq \norm{w_+}_{H^1(\Gamma)}.
\]
Combining these estimates gives that 
\begin{equation} \label{pltwo_norm_intermediate_estimate}
\norm{p}_{L^2(B)} \leq C ( \norm{w_+-w_-}_{H^1(\Gamma)} + \norm{w_+}_{H^1(\Sigma_{+})} + \norm{w_+}_{H^1(\Sigma_{+} \cap \Gamma)}).
\end{equation}

Next, to estimate the jump from $w_-$ to $w_+$ across $\Gamma$, we observe that for all $(y,z) \in \Bb$
\begin{align*}
 -2(w_+(y,z,z) - w_-(y,z,z)) & = \int_{-\infty}^z p(y,s) \,ds + \sigma \int_{-\infty}^{-z} p(Ay,s) \,ds \\
 &= \int_{-\infty}^{\infty} p(y,s) \,ds - \int_z^{\infty} (p(y,s) - \sigma p(Ay,-s)) \,ds \\
 &= -2 w_+(y, \sqrt{1-\abs{y}^2}, \sqrt{1-\abs{y}^2}) - 2 \int_z^{\infty} \breve{p}(y,s) \,ds.
\end{align*}
Writing $h(y,z) = \int_z^{\infty} \breve{p}(y,s) \,ds$, one has 
\[
w_+(y,z,z) - w_-(y,z,z) = w_+(P(y,z)) + h(y,z), \qquad \text{for all } ~ (y,z) \in \Bb,
\]
where $P: (y,z) \mapsto (y, \sqrt{1-\abs{y}^2}, \sqrt{1-\abs{y}^2})$ maps $\ol{B}$ to $\Sigma_+ \cap \Gamma$. It follows that 
\[
\norm{w_+-w_-}_{H^1(\Gamma)} \lesssim \norm{w_+}_{H^1(\Sigma_+ \cap \Gamma)} + \norm{h}_{H^1(B)}.
\]
Since $h(y,\sqrt{1-\abs{y}^2}) = 0$ for $\abs{y} \leq 1$, a simple Poincar\'e inequality implies that 
\[
\norm{h}_{H^1(B)} \lesssim \norm{\p_z h}_{H^1(B)} = \norm{\breve{p}}_{H^1(B)}.
\]
Inserting these facts in \eqref{pltwo_norm_intermediate_estimate}, we see that 
\[
\norm{p}_{L^2(B)} \leq C ( \norm{\breve{p}}_{H^1(B)} + \norm{w_+}_{H^1(\Sigma_{+})} + \norm{w_+}_{H^1(\Sigma_{+} \cap \Gamma)}).
\]
We now choose $\eps$ so small that $C\eps \leq 1/2$. If $p$ satisfies $\norm{\breve{p}}_{H^1(B)} \leq \eps \norm{p}_{L^2(B)}$,
the $\norm{\breve{p}}_{H^1(B)}$ term can be absorbed by the left hand side and the theorem follows.
\end{proof}

\begin{proof}[Proof of Proposition \ref{prop_single_measurement_perturbation}]
Let $\phi$ be the weight in Lemma \ref{lemma_pseudoconvex_weight}, so that $\phi$ is strongly pseudoconvex for 
$\Box$ in a neighborhood of $\ol{Q}$. We first use a Carleman estimate with boundary terms on
$Q_+$ (below we write $w$ and $q$ instead of $w_+$ and $q_+$ for convenience). By Theorem \ref{thm:carleman}, 
for $\sigma \geq \sigma_0 $ with $\sigma_0 \geq 1$ sufficiently large, one has the estimate 
\begin{align}
\sigma^3 \norm{e^{\sigma \phi} w}_{L^2(Q_+)}^2 + \sigma \norm{e^{\sigma \phi} \nabla w}_{L^2(Q_+)}^2 
 + \sigma \int_{\p Q_+} e^{2\sigma \phi}   F_j(x, & \sigma w, \nabla w) \nu_j \,dS 
 \label{stability_first_estimate_jump}
 \\
 & \lesssim \norm{e^{\sigma \phi} (\Box+q) w}_{L^2(Q_+)}^2.
 \nn
\end{align}
It is proved in Section \ref{subsection_boundary_terms_wave} that the functions $F_j(x, q_0, q_1, \ldots, q_{n+1})$ are 
quadratic forms in the $q_j$ variables with smooth coefficients depending on $x$. Moreover, it will be important that 
on $\Gamma$, a subset of $\pa Q_+$, the functions $F_j$ depend only on the tangential derivatives of $w$ and not
on the normal derivative of $w$ (see \eqref{carleman_boundary_terms_gamma}).

Now the energy estimate in Lemma \ref{lemma_energy_estimate_bottom} shows that 
\begin{multline}
\sigma^2 \norm{e^{\sigma \phi} w}_{L^2(\Gamma)}^2 + \norm{e^{\sigma \phi} \nabla_{\Gamma} w}_{L^2(\Gamma)}^2 \lesssim \sigma^3 \norm{e^{\sigma \phi} w}_{L^2(Q_+)}^2 \\
 + \sigma \norm{e^{\sigma \phi} \nabla w}_{L^2(Q_+)}^2 + \norm{e^{\sigma \phi} (\Box+q)w}_{L^2(Q_+)}^2 + \sigma^2 \norm{e^{\sigma \phi} w}_{L^2(\Sigma_+)}^2 + \norm{e^{\sigma \phi} \nabla w}_{L^2(\Sigma_+)}^2. \label{stability_second_estimate_jump}
\end{multline}
Combining \eqref{stability_first_estimate_jump} and \eqref{stability_second_estimate_jump} and dropping the $L^2(Q_+)$ terms on the left gives the estimate 
\begin{multline}
\sigma^2 \norm{e^{\sigma \phi} w}_{L^2(\Gamma)}^2 + \norm{e^{\sigma \phi} \nabla_{\Gamma} w}_{L^2(\Gamma)}^2 + \sigma \int_{\Gamma} e^{2\sigma \phi} F_j(x, \sigma w, \nabla_{\Gamma} w) \nu_j \,dS \\
 \lesssim \norm{e^{\sigma \phi} (\Box+q) w}_{L^2(Q_+)}^2 + \sigma^3 \norm{e^{\sigma \phi} w}_{L^2(\Sigma_+ \cup \Gamma_T)}^2 + \sigma \norm{e^{\sigma \phi} \nabla w}_{L^2(\Sigma_+ \cup \Gamma_T)}^2. \label{stability_third_estimate_jump}
\end{multline}
For the terms over $\Gamma_T$, using the energy estimate in Lemma \ref{lemma_energy_estimate_top}, we have 
\begin{align*}
 \norm{w}_{L^2(\Gamma_T)}^2 +  \norm{\nabla w}_{L^2(\Gamma_T)}^2 
 & \lesssim \norm{w}_{H^1(\Gamma)}^2 + \norm{(\Box+q) w}_{L^2(Q_+)}^2 + \norm{w}_{H^1(\Sigma_+)}^2 + \norm{\p_{\nu} w}_{L^2(\Sigma_+)}^2 \\
 &  \lesssim \norm{w}_{H^1(\Gamma)}^2 + \norm{w}_{H^1(\Sigma_+)}^2 + \norm{\p_{\nu} w}_{L^2(\Sigma_+)}^2
\end{align*}
In the last line we used that $(\Box+q) w = (Zw)|_{\Gamma} f_+$ with $f_+$ bounded. Since $\phi$ satisfies $\sup_{\Gamma_T} \phi \leq \inf_{\Gamma} \phi - \delta$ for some $\delta > 0$ (see Lemma \ref{lemma_pseudoconvex_weight}), we have 
\begin{align*}
\sigma^3 \norm{e^{\sigma \phi} w}_{L^2(\Gamma_T)}^2 + 
 \sigma & \norm{e^{\sigma \phi} \nabla w}_{L^2(\Gamma_T)}^2
\\
& \lesssim \sigma^3 e^{-2\delta \sigma} \norm{e^{\sigma \phi} w}_{H^1(\Gamma)}^2 
 + \sigma^3 e^{2\sigma \sup_{\Gamma_T} \phi} (\norm{w}_{H^1(\Sigma_+)}^2 + \norm{\p_{\nu} w}_{L^2(\Sigma_+)}^2).
\end{align*}
Inserting this estimate in \eqref{stability_third_estimate_jump}, and choosing $\sigma$ so large that the term with $\sigma^3 e^{-2\delta \sigma}$ can be absorbed on the left, we observe that 
\begin{multline}
\sigma^2 \norm{e^{\sigma \phi} w}_{L^2(\Gamma)}^2 + \norm{e^{\sigma \phi} \nabla_{\Gamma} w}_{L^2(\Gamma)}^2 + \sigma \int_{\Gamma} e^{2\sigma \phi} F_j(x, \sigma w, \nabla_{\Gamma} w) \nu_j \,dS \\
 \lesssim \norm{e^{\sigma \phi} (\Box+q) w}_{L^2(Q_+)}^2 + \sigma^3 e^{C\sigma} \left[ \norm{w}_{L^2(\Sigma_+)}^2 + \norm{\nabla w}_{L^2(\Sigma_+)}^2 \right]. \label{stability_fourth_estimate_jump}
\end{multline}
Again $(\Box+q) w = (Zw)|_{\Gamma} f_+$ with $f_+$ bounded, so 
\[
\norm{e^{\sigma \phi} (\Box+q) w}_{L^2(Q_+)}^2 \lesssim h(\sigma) \norm{e^{\sigma \phi} \nabla_{\Gamma} w}_{L^2(\Gamma)}^2,
\]
where $h(\sigma)$ is the function in Lemma \ref{lemma_pseudoconvex_weight} with $h(\sigma) \to 0$ as $\sigma \to \infty$. 
Thus, for $\sigma$ large (depending on $M$ and $T$), the $h(\sigma)$ term 
 can be absorbed on the left. Fixing such a $\sigma$, from \eqref{stability_fourth_estimate_jump} we obtain the estimate 
\begin{equation}
c \norm{w}_{H^1(\Gamma)}^2 + \sigma \int_{\Gamma} e^{2\sigma \phi} F_j(x, \sigma w, \nabla_{\Gamma} w) \nu_j \,dS \leq C (\norm{w}_{L^2(\Sigma_+)}^2 + \norm{\nabla w}_{L^2(\Sigma_+)}^2), \label{stability_fifth_estimate_jump}
\end{equation}
for some positive constants $c,C$ depending on $M,T$.

We rewrite the estimate \eqref{stability_fifth_estimate_jump} for $w = w_+$ as 
\begin{equation}
c \norm{w_+}_{H^1(\Gamma)}^2 + \sigma \int_{\Gamma} e^{2\sigma \phi} F_j(x, \sigma w_+, \nabla_{\Gamma} w_+) \nu_j \,dS \leq C (\norm{w_+}_{L^2(\Sigma_+)}^2 + \norm{\nabla w_+}_{L^2(\Sigma_+)}^2). \label{stability_sixth_estimate_jump}
\end{equation}
{\bf Fix $\nu$ to be the downward pointing unit normal to $\Gamma$}, so $\nu$ is an exterior normal for $\ol{Q}_+$.
An analogous argument in $Q_-$ yields the following estimate for $w_-$:
\begin{equation}
c \norm{w_-}_{H^1(\Gamma)}^2 - \sigma \int_{\Gamma} e^{2\sigma \phi} F_j(x, \sigma w_-, \nabla_{\Gamma} w_-) \nu_j \,dS \leq C (\norm{w_-}_{L^2(\Sigma_-)}^2 + \norm{\nabla w_-}_{L^2(\Sigma_-)}^2). \label{stability_seventh_estimate_jump}
\end{equation}
Note the negative sign in front of $\sigma$ in (\ref{stability_seventh_estimate_jump}) in comparison with
the positive sign in front of $\sigma$ in (\ref{stability_sixth_estimate_jump}); that is so because the $\nu$ we fixed is an 
interior normal for $Q_-$ on $\Gamma$.
Adding up \eqref{stability_sixth_estimate_jump} and \eqref{stability_seventh_estimate_jump} and noting
that the $F_j$ are quadratic forms in $\sigma w_{\pm}$ and $\nabla_{\Gamma} w_{\pm}$, we have 
\begin{align}
c \sum_{\pm} \norm{w_{\pm}}_{H^1(\Gamma)}^2 
\leq 
C \norm{w_+   - w_-}_{H^1(\Gamma)}& (   \norm{w_+}_{H^1(\Gamma)} + \norm{w_-}_{H^1(\Gamma)})
\label{stability_eight_estimate_jump}
\\
& + C \sum_{\pm} (\norm{w_{\pm}}_{L^2(\Sigma_{\pm})}^2 + \norm{\nabla w_{\pm}}_{L^2(\Sigma_{\pm})}^2),
\nn
\end{align}
for some positive constants $c, C$ depending on $\sigma$ (hence on $M$ and $T$). Using Cauchy's inequality with $\eps$ allows one to absorb the $\norm{w_{\pm}}_{H^1(\Gamma)}$ terms on the right into the terms on the left. This proves the proposition.
\end{proof}

\section{Horizontally controlled potentials} \label{sec_horizontally_controlled}

The following result is the time domain analogue of Theorem \ref{thm_main4} and also contains a Lipschitz stability estimate.

\begin{Theorem} \label{thm_single_measurement_perturbation2}
Let $M > 1$ and $T > 3$. There exist constants $C(M,T)>0$, $\eps(M,T)>0$ so that
if $q, p \in C^{\infty}_c(\mR^n)$ are supported in $\ol{B}$ and $\norm{q}_{C^{n+4}} \leq M$, $\norm{p}_{C^{n+4}} \leq M$,
then
\[
\norm{p}_{L^2(B)} \leq C (\norm{u_{q+p}-u_q}_{H^1(\Sigma_+)} + \norm{u_{q+p}-u_q}_{H^1(\Sigma_+ \cap \Gamma)})
\]
provided that the function 
\[
r(y,z) := \int_{-\infty}^z p(y,s) \,ds
\]
is horizontally $(M,\eps)$-controlled.
\end{Theorem}

The following lemma gives an example of a perturbation $p$ such that the corresponding function $r$ is $(M,\ep)$-controlled.

\begin{Lemma} \label{lemma_p_controlled}
Suppose $\varphi_1, \ldots, \varphi_R$ are linearly independent functions in $C^{\infty}_c(\mR^{n-1})$ supported in the ball of radius $1/\sqrt{2}$ and define
\begin{equation} \label{p_finite_sum}
p(y,z) := \sum_{j=1}^R p_j(z) \varphi_j(y)
\end{equation}
for some functions $p_j \in C^{\infty}_c(\mR)$ supported in $(-1/\sqrt{2}, 1/\sqrt{2})$. The function 
\[
r(y,z) := \int_{-\infty}^z p(y,s) \,ds
\]
is $(M,0)$-controlled for some $M$ depending on $R$ and $\varphi_1, \ldots, \varphi_R$.
\end{Lemma}
\begin{proof}
Note that $p$ is smooth and supported in $\ol{B}$. The function $r(y,z)$ has the form 
\[
r(y,z) = \sum_{j=1}^R r_j(z) \varphi_j(y), \qquad r_j(z) = \int_{-\infty}^z p_j(s) \,ds.
\]
By the triangle inequality 
\[
\int_{\mR^{n-1}} \abs{\nabla_y r(y,z)}^2 \,dy \lesssim \sum_{j=1}^R \ \abs{r_j(z)}^2,
\]
and moreover  
\[
\int_{\mR^{n-1}} \abs{r(y,z)}^2 \,dy = \sum_{j,k=1}^R r_j(z) r_k(z) (\varphi_j, \varphi_k)_{L^2(\mR^{n-1})} \sim \sum_{j=1}^R \ \abs{r_j(z)}^2
\]
since the matrix $((\varphi_j, \varphi_k)_{L^2(\mR^{n-1})})_{j,k=1}^R$ is positive definite by the linear independence of $\varphi_1, \ldots, \varphi_R$. Thus $r(y,z)$ is horizontally $(M,0)$-controlled for some $M$ depending on $R, \varphi_1, \ldots, \varphi_R$.
\end{proof}

Theorem \ref{thm_single_measurement_perturbation2} will be a consequence of the following proposition.

\begin{Proposition} \label{prop_single_measurement_perturbation2}
Let $M > 1$, $T > 3$. There are $C(M,T)$, $\eps(M,T) > 0$  so that if
\begin{align*}
(\Box + q) w(x,t) = (Zw)(x,z) f(x,t) \text{ in $Q_+$},
\end{align*}
for some $q \in C^{\infty}_c(\mR^n)$ supported in $\ol{B}$, $f \in L^{\infty}(Q_+)$ and $w \in H^2(Q_+)$ with 
$\norm{q}_{L^{\infty}(B)} \leq M$ and $\norm{f}_{L^{\infty}(Q_+)} \leq M$, then
\[
\norm{w}_{L^2(\Gamma)} + \norm{Zw}_{L^2(\Gamma)} \leq C ( \norm{w}_{H^1(\Sigma_+)} + \norm{\p_{\nu} w}_{L^2(\Sigma_+)})
\]
provided that the function $r(y,z) := w(y,z,z)$ is $(M,\eps)$-controlled.
\end{Proposition}

\begin{proof}[Proof of Theorem \ref{thm_single_measurement_perturbation2}]
Define 
\[
w := u_{q+p} - u_q.
\]
By Proposition \ref{prop_uq_basic}, the function $w$ is smooth in $\{ t \geq z \}$ and solves 
\[
(\Box + q)w = -p u_{q+p} \qquad \text{in $Q_+$},
\]
and $r(y,z) := w(y,z,z)$ is given by 
\[
r(y,z) = -\frac{1}{2} \int_{-\infty}^z p(y,s) \,ds.
\]
In particular, 
\begin{equation} \label{zw_formula}
Zw(y,z,z) = \frac{1}{\sqrt{2}} \p_z(w(y,z,z)) = -\frac{1}{2 \sqrt{2}} p(y,z).
\end{equation}
We may thus use Proposition \ref{prop_single_measurement_perturbation2} with the choice $f(x,t) := 2 \sqrt{2} u_{q+p}(x,t)$, and with some new choice of $M$, to obtain that 
\begin{equation} \label{zw_estimate}
\norm{w}_{L^2(\Gamma)} + \norm{Zw}_{L^2(\Gamma)} \leq C ( \norm{w}_{H^1(\Sigma_+)} + \norm{\p_{\nu} w}_{L^2(\Sigma_+)})
\end{equation}
where $C$ only depends on $M$ and $T$. By Lemma \ref{lemma_energy_estimate_exterior} we have 
\begin{equation} \label{pnuzw_estimate}
\norm{\p_{\nu} w}_{L^2(\Sigma_+)} \leq C(\norm{w}_{H^1(\Sigma_+)} + \norm{w}_{H^1(\Sigma_+ \cap \Gamma)}).
\end{equation}
Theorem \ref{thm_single_measurement_perturbation2} follows by combining \eqref{zw_estimate}, \eqref{zw_formula} and \eqref{pnuzw_estimate}.
\end{proof}

The proof of Proposition \ref{prop_single_measurement_perturbation2} is again based on a Carleman estimate. However, 
in this case, it is convenient to use a weight $\phi$ that is independent of $y$ and satisfies $N\phi|_{\Gamma} > 0$, 
$\p_t \phi|_{Q_+} \leq 0$. The following lemma gives one such weight.

\begin{Lemma} \label{lemma_carleman_weight_y_controlled}
For any $T > 3$ there exist $a > b \geq T$ so that if one defines 
\[
\psi(y,z,t) := \frac{1}{2} ( (z-a)^2 + (t-b)^2 ),
\]
then, for $\lambda > 0$ sufficiently large, the function 
\[
\phi(y,z,t) := e^{\lambda \psi(y,z,t)}
\]
is strongly pseudoconvex for $\Box$ in a neighborhood of $\ol{Q}$. Moreover, 
\[
N\phi|_{\Gamma} > 0, \qquad Z \phi|_{\Gamma} < 0, \qquad \p_t \phi|_{Q} \leq 0,
\]
the smallest value of $\phi$ on $\Gamma$ is strictly larger than the largest value of $\phi$ on $\Gamma_T$, 
and
\[
g_{\sigma}(y,z) := \int_z^T e^{2\sigma(\phi(y,z,t) - \phi(y,z,z))} \,dt \leq T+1,
\]
uniformly over $\sigma \geq 1$ and $(y,z) \in \ol{B}$.
\end{Lemma}
\begin{proof}
Let $a > b \geq T > 3$. Note first that $\p_z \psi = z-a \neq 0$ whenever $\abs{z} \leq 1$, showing that $\nabla \psi$ is nonvanishing near $\ol{Q}$. The symbol of $\Box$ is 
\[
p(y,z,t,\eta,\zeta,\tau) = -\tau^2 + \abs{\eta}^2 + \zeta^2.
\]
Since $\psi$ only depends on $z$ and $t$, we compute 
\begin{align*}
\{ p, \psi \} &= 2 \zeta (z-a) - 2 \tau (t-b), \\
\{ p, \{ p, \psi \} \} &= (2 \zeta) (2 \zeta) + (2 \tau) (2 \tau) = 4(\zeta^2 + \tau^2).
\end{align*}
Thus always $\{ p, \{ p, \psi \} \} \geq 0$. If one has $\{ p, \{ p, \psi \} \}(y,z,t,\eta,\zeta,\tau) = 0$ at some point where $p = 0$, then $\zeta = \tau = 0$ and hence $p = \abs{\eta}^2 = 0$, showing that $\eta = \zeta = \tau = 0$. This proves that $\{ p, \{ p, \psi \} \} > 0$ whenever $p = \{ p, \psi \} = 0$ and $(\eta,\zeta,\tau) \neq 0$, and thus the level surfaces of $\psi$ are pseudoconvex for $\Box$. Combining Propositions \ref{prop:level} and \ref{prop:special}, it follows that $\phi$ is strongly pseudoconvex for $\Box$ near $\ol{Q}$ if $\lambda > 0$ is sufficiently large.

Now take $T > 3$ and compute 
\begin{align*}
\sqrt{2} N\psi|_{\Gamma} &= t-b - (z-a)|_{\Gamma} = a-b, \\
\sqrt{2} Z\psi|_{\Gamma} &= t-b + (z-a)|_{\Gamma} \leq 2 - a - b,
\end{align*}
with
\[
\p_t \psi|_{Q} = t-b|_{Q} \leq T-b.
\]
Thus $N\phi|_{\Gamma} > 0$, $Z\phi|_{\Gamma} < 0$ and $\p_t \phi|_{Q} \leq 0$ whenever $a > b \geq T > 3$. On $\Gamma$ we have 
\[
\psi(y,z,z) = \frac{1}{2}( (z-a)^2 + (z-b)^2 ) \geq \frac{1}{2}( (1-a)^2 + (1-b)^2 )
\]
since $\abs{z} \leq 1$ and $a, b \geq 1$. On $\Gamma_T$ we have 
\[
\psi(y,z,T) = \frac{1}{2}( (z-a)^2 + (T-b)^2 ) \leq \frac{1}{2}( (a+1)^2 + (T-b)^2 ).
\]
Comparing the two values on the right, we have 
\[
(1-a)^2 + (1-b)^2 - \left[ (a+1)^2 + (T-b)^2 \right] = - T^2 + 2bT - 4a - 2b + 1.
\]
Given $T > 3$, we want to choose $a > b \geq T$ so that the expression on the right is positive. Choosing $a > b$ but $a$ 
very close to $b$, it is enough to choose $b \geq T$ so that 
\[
-T^2 + (2T - 6)b + 1 > 0.
\]
Since $T > 3$, it is enough to choose $b$ so that $b > \frac{T^2-1}{2T-6}$ and $b \geq T$.

With the above choices, we have proved everything except for the claim about $g_{\sigma}$. However, since $\p_t \phi|_{Q} \leq 0$, the integrand in $g_{\sigma}$ is $\leq 1$ and hence $g_{\sigma}|_{\ol{B}} \leq T+1$ uniformly in $\sigma$.
\end{proof}

\begin{proof}[Proof of Proposition \ref{prop_single_measurement_perturbation2}]
Let $\phi$ be as in Lemma \ref{lemma_carleman_weight_y_controlled}. Repeating the argument in Proposition \ref{prop_single_measurement_perturbation} (but using Lemma \ref{lemma_carleman_weight_y_controlled} for the 
properties of $\phi$), we arrive at the estimate \eqref{stability_fourth_estimate_jump}, which we restate below except that 
we write the integrand on $\Gamma$ as $\nu^j E^j$ as in Theorem \ref{thm:carleman}. So, for any $\sigma \geq \sigma_0$ with 
$\sigma_0$ large enough, we have 
\begin{multline}
\sigma^2 \norm{e^{\sigma \phi} w}_{L^2(\Gamma)}^2 + \norm{e^{\sigma \phi} \nabla_{\Gamma} w}_{L^2(\Gamma)}^2 + \sigma \int_{\Gamma} \nu^j E^j \,dS \\
 \lesssim \norm{e^{\sigma \phi} (\Box+q) w}_{L^2(Q_+)}^2 + \sigma^3 e^{C\sigma} \left[ \norm{w}_{L^2(\Sigma_+)}^2 + \norm{\nabla w}_{L^2(\Sigma_+)}^2 \right],
\label{stability_fourth_estimate_jump2}
\end{multline}
with constants depending only on $M$ and $T$. 
Since $(\Box+q)w = Zw|_{\Gamma} f$ where $\norm{f}_{L^{\infty}} \leq M$, one has 
\[
\norm{e^{\sigma \phi} (\Box+q) w}_{L^2(Q_+)} \leq M \norm{e^{\sigma (\phi-\phi|_{\Gamma})} (e^{\sigma \phi} Zw)|_{\Gamma}}_{L^2(Q_+)} \leq M \norm{g_{\sigma} e^{\sigma \phi} Zw}_{L^2(\Gamma)}.
\]
By Lemma \ref{lemma_carleman_weight_y_controlled}, the function $g_{\sigma}$ is bounded uniformly over $\sigma$,
 hence one has $\norm{e^{\sigma \phi} (\Box+q) w}_{L^2(Q_+)} \leq C \norm{e^{\sigma \phi} Zw}_{L^2(\Gamma)}$ with 
 $C = C(M,T)$. Thus \eqref{stability_fourth_estimate_jump2} gives 
\begin{multline}
\sigma^2 \norm{e^{\sigma \phi} w}_{L^2(\Gamma)}^2 + \norm{e^{\sigma \phi} \nabla_{\Gamma} w}_{L^2(\Gamma)}^2 + \sigma \int_{\Gamma} \nu^j E^j \,dS \\
 \lesssim \norm{e^{\sigma \phi} Zw}_{L^2(\Gamma)}^2 + \sigma^3 e^{C\sigma} \left[ \norm{w}_{L^2(\Sigma_+)}^2 + \norm{\nabla w}_{L^2(\Sigma_+)}^2 \right]. \label{stability_fifth_estimate_jump2}
\end{multline}

At this point we study the integral over $\Gamma$ in (\ref{stability_fifth_estimate_jump2}).
Now $\phi$ is independent of $y$ and 
\[
N\phi|_{\Gamma} > 0, \qquad Z\phi|_{\Gamma} < 0
\]
by Lemma \ref{lemma_carleman_weight_y_controlled}. Hence, using the expressions for $E^j$ in 
\eqref{carleman_boundary_terms_gamma}, we have
\begin{equation} \label{carleman_boundary_terms_estimate1}
\sigma \int_{\Gamma} \nu^j E^j \,dS \geq c \sigma \int_{\Gamma} ( (Zv)^2 + \sigma^2 v^2 ) \,dS 
- C \sigma \int_{\Gamma} ( \abs{\nabla_y v}^2 + \abs{v}\,\abs{Zv} ) \,dS ,
\end{equation}
for some positive $c$, $C$ independent of $\sigma$; note that $v = e^{\sigma \phi} w$. Since
\[
Zv = e^{\sigma \phi}(Zw + \sigma (Z\phi) w),
\]
for every $r > 0$ we have
\begin{align*}
\norm{e^{\sigma \phi} Zw}_{L^2(\Gamma)}^2 &= \norm{Zv - e^{\sigma \phi} \sigma (Z\phi) w}_{L^2(\Gamma)}^2 \\
 &\leq (1+r) \norm{Zv}_{L^2(\Gamma)}^2 + (1+1/r) \norm{e^{\sigma \phi} \sigma (Z\phi) w}_{L^2(\Gamma)}^2.
\end{align*}
Taking $\beta := \frac{1}{1+r} \in (0,1)$, so $\frac{1}{r} = \frac{\beta}{1-\beta}$, we have 
\[
\norm{Zv}_{L^2(\Gamma)}^2 \geq \beta \norm{e^{\sigma \phi} Zw}_{L^2(\Gamma)}^2 - \frac{\beta}{1-\beta} \norm{e^{\sigma \phi} \sigma (Z\phi) w}_{L^2(\Gamma)}^2.
\]
Using this estimate  in \eqref{carleman_boundary_terms_estimate1} with sufficiently small $\beta \in (0,1)$, together with 
$2 ab < \ep a^2 + \ep^{-1} b^2$ for $\eps>0$, for $\sigma$ sufficiently large one has 
\[
\sigma \int_{\Gamma} \nu^j E^j \,dS \geq c \sigma \int_{\Gamma} e^{2\sigma \phi}( (Zw)^2 + \sigma^2 w^2 ) \,dS - C \sigma \int_{\Gamma} e^{2\sigma \phi} \abs{\nabla_y w}^2 \,dS.
\]
Inserting this in \eqref{stability_fifth_estimate_jump2} leads to 
\begin{multline*}
\sigma^3 \norm{e^{\sigma \phi} w}_{L^2(\Gamma)}^2 + \sigma \norm{e^{\sigma \phi} Zw}_{L^2(\Gamma)}^2 \lesssim \norm{e^{\sigma \phi} Zw}_{L^2(\Gamma)}^2 \\
 + \sigma \norm{e^{\sigma \phi} \nabla_{y} w}_{L^2(\Gamma)}^2  + \sigma^3 e^{C\sigma} \left[ \norm{w}_{L^2(\Sigma_+)}^2 + \norm{\nabla w}_{L^2(\Sigma_+)}^2 \right],
\end{multline*}
which, when compared to (\ref{stability_fourth_estimate_jump}),  has improved powers of $\sigma$ on the left hand side 
but with a $\nabla_y w$ term on the right hand side.
Choosing $\sigma$ large enough, we may absorb the $\norm{e^{\sigma \phi} Zw}_{L^2(\Gamma)}^2$ term into the left 
side, hence 
\begin{equation}
\sigma^3 \norm{e^{\sigma \phi} w}_{L^2(\Gamma)}^2 + \sigma \norm{e^{\sigma \phi} Zw}_{L^2(\Gamma)}^2 \lesssim \sigma \norm{e^{\sigma \phi} \nabla_{y} w}_{L^2(\Gamma)}^2 
 + \sigma^3 e^{C\sigma} \left[ \norm{w}_{L^2(\Sigma_+)}^2 + \norm{\nabla w}_{L^2(\Sigma_+)}^2 \right]. \label{stability_sixth_estimate_jump2}
\end{equation}

Now $\phi$ is independent of $y$, so invoking the assumption that $r(y,z) := w(y,z,z)$ is $(M,\eps)$-controlled ($\eps$ still to be determined) leads to the estimate 
\[
\sigma \norm{e^{\sigma \phi} \nabla_{y} w}_{L^2(\Gamma)}^2 \leq M \sigma \norm{e^{\sigma \phi} w}_{L^2(\Gamma)}^2 + \eps \sigma \norm{e^{\sigma \phi} Z w}_{L^2(\Gamma)}^2.
\]
Using this in \eqref{stability_sixth_estimate_jump2}, 
choosing $\eps(M,T)>0$ small enough and $\sigma$ large enough, we may absorb the 
$\eps \sigma \norm{e^{\sigma \phi} Z w}_{L^2(\Gamma)}^2$ term and the 
$M \sigma \norm{e^{\sigma \phi} w}_{L^2(\Gamma)}^2$ term
into the left hand side of \eqref{stability_sixth_estimate_jump2}.
So fixing a large enough $\sigma$ and letting all constants depend on $\sigma$, we obtain
\[
\norm{w}_{L^2(\Gamma)}^2 + \norm{Zw}_{L^2(\Gamma)}^2 \lesssim \norm{w}_{L^2(\Sigma_+)}^2 + \norm{\nabla w}_{L^2(\Sigma_+)}^2.
\]
This proves the proposition.
\end{proof}

\section{Equivalence of frequency and time domain problems} \label{sec_equivalence_frequency_time}

The following theorem shows that the scattering amplitude for a fixed direction $\omega \in S^{n-1}$ and the boundary measurements in the wave equation problem in Section \ref{sec_time_domain_setting} are equivalent information. Related results in the context of Lax-Phillips scattering theory in odd dimensions $n \geq 3$ are discussed in \cite{Melrose, Uhlmann_backscattering, MelroseUhlmann_bookdraft}. We write $u_q(x,t,\omega)$ for the solution in Proposition \ref{prop_uq_basic}, where $e_n$ is replaced by $\omega$, so that $u_q(x,t,\omega)$ is smooth in $\{ t \geq x \cdot \omega \}$.

\begin{Theorem} \label{thm_frequencydomain_reduction}
Let $n \geq 2$ and fix $\omega \in S^{n-1}$, $\lambda_0 > 0$. For any real valued $q_1, q_2 \in C^{\infty}_c(\mR^n)$
with support in $\ol{B}$, one has
\[
a_{q_1}(\lambda, \theta, \omega) = a_{q_2}(\lambda, \theta, \omega) \text{ for $\lambda \geq \lambda_0$ and $\theta \in S^{n-1}$}
\]
if and only if  
\[
u_{q_1}(x,t,\omega) = u_{q_2}(x, t, \omega) \text{ for $(x,t) \in (S \times \mR) \cap \{�t \geq x \cdot \omega \}$.}
\]
\end{Theorem}

Given the previous result, Theorem \ref{thm_main1} and Corollary \ref{thm_main2} in the introduction follow immediately from \cite[Theorem 1.2]{RakeshSalo} and \cite[Corollary 1.3]{RakeshSalo}, respectively. In a similar way, Theorems \ref{thm_main3} and \ref{thm_main4} follow from Theorems \ref{thm_single_measurement_perturbation1} and \ref{thm_single_measurement_perturbation2}, respectively.

We first give a formal argument explaining why Theorem \ref{thm_frequencydomain_reduction} could be true. It will be convenient to use the slightly nonstandard conventions 
\[
\tilde{f}(\lambda) = \int_{-\infty}^{\infty} e^{i\lambda t} f(t) \,dt, \qquad \breve{F}(t) = \frac{1}{2\pi} \int_{-\infty}^{\infty} e^{-i\lambda t} F(\lambda) \,d\lambda,
\]
for the Fourier transform and its inverse for Schwartz functions (and via extension also for tempered distributions) on the real line.

Let $q \in C^{\infty}_c(\mR^n)$ be supported in $\ol{B}$, and let $U_q(x,t,\omega)$ solve 
\[
(\p_t^2 - \Delta + q(x)) U_q = 0 \text{ in $\mR^n \times \mR$}, \qquad U_q|_{\{t < -1\}} = \delta(t-x\cdot \omega).
\]
Then of course $u_q = U_q - \delta(t-x \cdot \omega)$. Suppose for the moment that the Fourier transform of $U_q$ in the time variable is well defined. The function $\tilde{U}_q$ should then solve for each $\lambda \in \mR$ the equation 
\[
(-\Delta+q(x)-\lambda^2) \tilde{U}_q(x,\lambda) = 0 \text{ in $\mR^n$}.
\]
One has $\tilde{U}_q(x,\lambda) = e^{i\lambda x \cdot \omega} + \tilde{u}_q(x,\lambda)$ where $\tilde{u}_q(x,\lambda)$ extends holomorphically to $\{ \im(\lambda) > 0 \}$ since $u_q$ vanishes for $t < -1$. These are exactly the properties that characterize the outgoing eigenfunction $\psi_q(x,\lambda,\omega)$ discussed in Section \ref{sec_introduction}, and thus one might expect that 
\[
\tilde{U}_q(x,\lambda,\omega) = \psi_q(x, \lambda, \omega).
\]

Now, the condition $a_{q_1}(\lambda, \,\cdot\,, \omega) = a_{q_2}(\lambda, \,\cdot\,, \omega)$ implies (by the Rellich uniqueness theorem, see e.g.\ \cite{Hormander_rellich}) that the outgoing eigenfunctions for $q_1$ and $q_2$ agree outside the support of the potentials:
\begin{equation} \label{outgoing_eigenfunction_exterior}
\psi_{q_1}(\,\cdot\,, \lambda, \omega)|_{\mR^n \setminus \ol{B}} = \psi_{q_2}( \,\cdot\,, \lambda, \omega)|_{\mR^n \setminus \ol{B}}.
\end{equation}
If the map $\lambda \mapsto \psi_{q_j}(x,\lambda,\omega)$ were smooth near $\lambda = 0$, then one would have \eqref{outgoing_eigenfunction_exterior} for all $\lambda \in \mR$. Taking the inverse Fourier transform in $\lambda$ would imply that 
\[
U_{q_1}(\,\cdot\,, t,\omega)|_{\mR^n \setminus \ol{B}} = U_{q_2}(\,\cdot\,,t,\omega)|_{\mR^n \setminus \ol{B}}.
\]
This would show that the boundary measurements for the wave equation problem, for a plane wave traveling in direction $\omega$, agree for $q_1$ and $q_2$.

The argument above is only formal, since it requires taking Fourier transforms in time and needs 
the regularity of the map $\lambda \mapsto \psi_q(x, \lambda, \omega)$ on the real line. The regularity of this map is related to the poles of the meromorphic continuation of the resolvent $(-\Delta + q - \lambda^2)^{-1}$ initially defined in $\{ \im(\lambda) > 0 \}$. It is well known \cite{Melrose} that the resolvent family has at most finitely many poles in $\{ \im(\lambda) > 0 \}$, located at $i r_1, \ldots, i r_N$ where $-r_1^2, \ldots, -r_N^2$ are the negative eigenvalues of $-\Delta + q$. Moreover, there may be a pole at $\lambda = 0$ corresponding to a bound state or resonance at zero energy. Such poles do not exist in $\{ \im(\lambda) \geq 0 \}$ if $q \geq 0$, but for signed potentials they can exist and thus the argument above does not work in general.

We now give a rigorous proof of Theorem \ref{thm_frequencydomain_reduction}, working on the set $\{ \im(\lambda) > r \}$, where the resolvent family has no poles, and using the Laplace transform in time instead of the Fourier transform. We first recall a few basic facts about the resolvent family. Below $\mC_+ := \{ \lambda \in \mC \,;\, \im(\lambda) > 0 \}$.

\begin{Proposition} \label{prop_resolvent_meromorphic}
Let $q \in C^{\infty}_c(\mR^n)$ be real valued and let $r_0 = \max(-\inf q, 0 )^{1/2}$. For any $\lambda \in \mC_+ \setminus i(0,r_0]$, there is a bounded operator 
\[
R_q(\lambda): L^2(\mR^n) \to L^2(\mR^n)
\]
such that for any $f \in L^2(\mR^n)$, the function $u = R_q(\lambda) f$ is the unique solution in $L^2(\mR^n)$ of 
\[
(-\Delta + q -\lambda^2) u = f \text{ in $\mR^n$}.
\]
For any fixed $r > r_0$, one has 
\[
\norm{R_q(\lambda)}_{L^2 \to L^2} \leq C_{r,q}, \qquad \im(\lambda) \geq r.
\]
For any fixed $\rho \in C^{\infty}_c(\mR^n)$ and for $\lambda$ in the set $\mC_+ \setminus i(0,r_0]$, the family 
\[
(\rho R_q(\lambda) \rho)_{\lambda \in \mC_+ \setminus i(0,r_0]}
\]
is a holomorphic family of bounded operators on $L^2(\mR^n)$ that extends continuously to $\ol{\mC}_+ \setminus i[0,r_0]$.
\end{Proposition}
\begin{proof}
The operator $-\Delta + q$, with domain $H^2(\mR^n)$, is a self-adjoint unbounded operator in $L^2(\mR^n)$ with spectrum contained in $[-r_0^2,\infty)$. If $\lambda \in \mC_+ \setminus i(0,r_0]$, then $\lambda^2$ is away from the spectrum, and one can choose $R_q(\lambda)$ to the be standard $L^2$ resolvent $(-\Delta+q-\lambda^2)^{-1}$. One has the estimate 
\[
\norm{R_q(\lambda)}_{L^2 \to L^2} \leq \frac{1}{\mathrm{dist}(\lambda^2, [-r_0^2, \infty))}.
\]
Writing $\lambda = \sigma + i\mu$, the range of $\sigma \mapsto (\sigma + i\mu)^2$ is a parabola opening to the right, so its distance from the spectrum is at least $2\sigma \mu$ when 
$\sigma^2 \geq \frac{1}{2}(\mu^2 - r_0^2)$ and at least $ \mu^2-\sigma^2-r_0^2$ when $\sigma^2 \leq \frac{1}{2}(\mu^2 - r_0^2)$. Thus, for $\im(\lambda) \geq r > r_0$, one has 
$\mathrm{dist}(\lambda^2, [-r_0^2,\infty)) \geq c > 0$ for some constant $c$ depending on $r$ and $r_0$ (in fact the distance is $\geq c (1+\abs{\sigma})$). It follows that 
\[
\norm{R_q(\lambda)}_{L^2 \to L^2} \leq C_{r,q}.
\]

The last statement follows from the meromorphic extension of the resolvent family from $\{ \im(\lambda) > 0 \}$ to $\mC$ (resp.\ a logarithmic cover of $\mC \setminus \{ 0 \}$) if $n$ is odd (resp.\ if $n$ is even), and from the fact that the only poles of this family in $\{ \im(\lambda) \geq 0 \}$ are in $i[0,r_0]$. See \cite{DyatlovZworski} for the case of odd dimensions, and \cite{Melrose} for the general case (note that \cite{Melrose} uses the opposite convention of extending the resolvent family from $\{ \im(\lambda) < 0 \}$). Here we only need the continuous extension of the resolvent family up to the real axis minus the origin (i.e.\ the limiting absorption principle), so we do not need to worry about the behaviour of the extension beyond the real axis.
\end{proof}

We also recall the following fact about Fourier-Laplace transforms.

\begin{Lemma} \label{lemma_fourier_laplace_inverse}
Suppose $F(z)$ is analytic on $\{ \im(z) > r \}$ for some $r \in \mR$ and
\[
\abs{F(z)} \leq C(1+\abs{z})^N e^{R \,\im(z)}, \qquad \text{for} ~ \im(z) > r,
\]
for some positive $R,C,N$ independent of $z$.
There exists an $f \in \mathcal{D}'(\mR)$ with $\supp(f) \subset [-R,\infty)$ and
$e^{-(\mu-r) t} f \in \mathcal{S}'(\mR)$, $(e^{-(\mu-r) t} f)\etilde(\,\cdot\,) = F(\,\cdot\,+i \mu)$,
for every $\mu>r$.
\end{Lemma}
\begin{proof}
Here $\hat{f}(\lambda) = \tilde{f}(-\lambda)$ will be the Fourier transform of $f$ following the convention in 
\cite[Section 7.1]{Hormander}. Define
\[
U(z) := e^{-iRz} F(-(z-ir)), \qquad \im(z)<0;
\]
then $U(z)$ is analytic on $\{ \im(z) < 0 \}$ and, on this set,
\[
\abs{U(z)} \leq e^{R \,\im(z)} C (1+\abs{z-ir})^N e^{R(r-\im(z))} \leq C_{r,R,N} (1+\abs{z})^N
\]
for some $C,N$ independent of $z$. Hence, from \cite[Section 7.4]{Hormander}, there is a
 $u \in \mathcal{D}'(\mR)$ with $\supp(u) \subset [0,\infty)$ and 
$(e^{-\eta t} u)\ehat(\sigma) = U(\sigma-i\eta)$ for every $\eta > 0$, $\sigma \in \R$. 
Define $f \in \mathcal{D}'(\mR)$ by
\[
f(\cdot) := u(\,\cdot\, + R);
\]
then $\supp(f) \subset [-R,\infty)$ and, for every $\eta > 0$, we have 
\begin{align*}
(e^{-\eta t} f) \etilde (\sigma) &= (e^{-\eta t} u(\,\cdot\,+R))\ehat(-\sigma) = e^{R(\eta-i\sigma)} (e^{-\eta t} u)\ehat(-\sigma) \\
 &= e^{R(\eta-i\sigma)} U(-\sigma-i\eta) = F(\sigma + i(\eta+r)).
\end{align*}
The result follows by taking $\eta = \mu-r$ for any $\mu > r$.
\end{proof}

The next result gives a precise relation between the time domain and frequency domain measurements. We write $\langle u, \varphi \rangle$ for the distributional pairing of $u$ and $\varphi$.

\begin{Proposition} \label{prop_uq_psiq_relation}
Suppose $\omega \in S^{n-1}$ is fixed and $q \in C^{\infty}_c(\mR^n)$ is real valued and supported in $\ol{B}$.
Define $r_0 := \max(-\inf q, 0 )^{1/2}$ and
\[
\psi_q^s(\,\cdot\,,\lambda,\omega) := R_q(\lambda)(-q e^{i\lambda x \cdot \omega}), \qquad \lambda \in 
\mC_+ \setminus i(0,r_0].
\]
We have 
\[
\langle \,u_q(x,t,\omega), \,\varphi(x) \chi(t) \,\rangle_{\mR^n_x \times \mR_t} = \langle \,\psi_q^s(x,\sigma+i\mu, \omega), \,\varphi(x) (e^{\mu t} \chi)\ebreve(\sigma) \,\rangle_{\mR^n_x \times \mR_{\sigma}},
\]
for all $\mu > r_0$ and all $\varphi \in C^{\infty}_c(\mR^n)$,  $\chi \in C^{\infty}_c(\mR)$.
\end{Proposition}

\begin{Remark}
Recall that by the Schwartz kernel theorem, any distribution $u(x,t)$ on $\mR^n \times \mR$ is uniquely determined by the 
values of
$\langle \, u(x,t), \varphi(x) \chi(t) \,\rangle_{\mR^n_x \times \mR_t}$ as $\varphi$ varies over $ C^{\infty}_c(\mR^n)$ and 
$\chi$ varies over $C^{\infty}_c(\mR)$ (see \cite[proof of Theorem 5.1.1]{Hormander}). The relation in Proposition \ref{prop_uq_psiq_relation} may be formally interpreted as an inverse Laplace transform identity 
\[
u_q(x,t,\omega) = \frac{1}{2\pi} \int_{\im(\lambda) = \mu} e^{-i\lambda t} \psi_q^s(x,\lambda,\omega) \,d\lambda
\]
when $\mu > r_0$.
\end{Remark}

\begin{proof}[Proof of Proposition \ref{prop_uq_psiq_relation}]
Fix $r > r_0$. By Proposition \ref{prop_resolvent_meromorphic}, for $\im(\lambda) \geq r$, one has the estimates 
\begin{equation} \label{psiqs_ltwo_rough_estimates}
\norm{\psi_q^s(\,\cdot\,,\lambda,\omega)}_{L^2} \leq C_{r,q} \norm{q e^{-\im(\lambda) x \cdot \omega}}_{L^2} \leq C_{r,q} e^{\im(\lambda)}.
\end{equation}
For any $\varphi \in C^{\infty}_c(\mR^n)$, define 
\[
F_{\varphi}(\lambda) = \int_{\mR^n} \psi_q^s(x,\lambda,\omega) \varphi(x) \,dx \qquad \text{for } ~\im(\lambda) \geq r.
\]
By Proposition \ref{prop_resolvent_meromorphic}, $F_{\varphi}$ is analytic on $\{ \im(\lambda) > r \}$ and 
\[
\abs{F_{\varphi}(\lambda)} \leq C_{r,q} e^{\im(\lambda)} \norm{\varphi}_{L^2}, \qquad \im(\lambda) \geq r.
\]
Using Lemma \ref{lemma_fourier_laplace_inverse}, there is a $f_{\varphi} \in \mathcal{D}'(\mR)$ with
 $\supp(f_{\varphi}) \subset [-1,\infty)$ and $(e^{-(\mu-r) t} f_{\varphi})\etilde(\,\cdot\,) = F_{\varphi}(\,\cdot\,+i \mu)$ 
 for every $\mu > r$. This means that 
\[
\langle e^{-(\mu-r)t} f_{\varphi}, \chi \rangle = \langle F_{\varphi}(\,\cdot\,+i\mu), \breve{\chi} \rangle, \qquad \chi \in C^{\infty}_c(\mR).
\]

Now, given $\mu > r$, define the linear map 
\[
\mathcal{K}: C^{\infty}_c(\mR^n) \to \mathcal{D}'(\mR), \ \ \mathcal{K} \varphi = e^{-(\mu-r) t} f_{\varphi}.
\]
If $\varphi_j \to 0$ in $C^{\infty}_c(\mR^n)$, then $F_{\varphi_j}(\lambda) \to 0$ when $\im(\lambda) \geq r$, which implies that 
\[
\langle e^{-(\mu-r) t} f_{\varphi_j}, \chi \rangle_{\mR_t} = \langle F_{\varphi_j}(\,\cdot\,+i \mu), \breve{\chi} \rangle_{\mR} \to 0
\]
as $j \to \infty$ for any fixed $\chi \in C^{\infty}_c(\mR)$. Thus $\mathcal{K}$ is continuous, and by the Schwartz kernel theorem \cite[Theorem 5.2.1]{Hormander} there is a unique $K \in \mathcal{D}'(\mR^n \times \mR)$ so that 
\begin{align} \label{k_varphi_chi_formula}
\langle K, \varphi(x) \chi(t) \rangle &= \langle \mathcal{K}\varphi, \chi \rangle = \langle e^{-(\mu-r) t} f_{\varphi}, \chi \rangle = \langle F_{\varphi}(\,\cdot\,+i\mu), \breve{\chi} \rangle \notag \\
 &= \langle \psi_q^s(x, \sigma+i\mu, \omega), \varphi(x) \breve{\chi}(\sigma) \rangle_{\mR^n_x \times \mR_{\sigma}}.
\end{align}
Since $f_{\varphi}$ is supported in $[-1,\infty)$, it follows that $K$ is supported in $\{ t \geq -1 \}$. We define 
\[
v_q(x,t) = e^{\mu t} K(x,t) \in \mathcal{D}'(\mR^n \times \mR).
\]
Then also $v_q$ is supported in $\{ t \geq -1 \}$.

If we show that 
\begin{equation} \label{boxqvq_equation}
(\Box + q) v_q = -q \delta(t-x\cdot \omega) \text{ in $\mR^n \times \mR$},
\end{equation}
uniqueness of distributional solutions of the wave equation supported in $\{ t \geq -1 \}$ (see e.g.\ \cite[Theorem 23.2.7]{Hormander}) implies that $u_q = v_q$, so 
\begin{align*}
\langle u_q, \varphi(x) \chi(t) \rangle &= \langle K, \varphi(x) e^{\mu t} \chi(t) \rangle \\
 &= \langle \psi_q^s(x, \sigma+i\mu, \omega), \varphi(x) (e^{\mu t} \chi)\ebreve(\sigma) \rangle_{\mR^n_x \times \mR_{\sigma}},
\end{align*}
proving the proposition.

To show \eqref{boxqvq_equation}, we first use \eqref{k_varphi_chi_formula} to see that 
\begin{align*}
\langle \p_t^j K, \varphi(x) \chi(t) \rangle &= \langle K, \varphi(x) (-\p_t)^j \chi(t) \rangle \\
 &= \langle (-i \sigma)^j \psi_q^s(x, \sigma+i\mu, \omega), \varphi(x) \breve{\chi}(\sigma) \rangle_{\mR^n_x \times \mR_{\sigma}}.
\end{align*}
Similarly 
\[
\langle \Delta_x K, \varphi(x) \chi(t) \rangle = \langle \Delta_x \psi_q^s(x, \sigma+i\mu, \omega), \varphi(x) \breve{\chi}(\sigma) \rangle_{\mR^n_x \times \mR_{\sigma}}
\]
and 
\[
\langle q(x) K, \varphi(x) \chi(t) \rangle = \langle  q(x) \psi_q^s(x, \sigma+i\mu, \omega), \varphi(x) \breve{\chi}(\sigma) \rangle_{\mR^n_x \times \mR_{\sigma}}
\]
Thus, since $v_q = e^{\mu t} K$, we obtain that 
\begin{align*}
 &\langle (\p_t^2 - \Delta_x + q) v_q, \varphi(x) \chi(t) \rangle \\
 & \hspace{20pt}= \langle (\p_t^2 +2\mu \p_t + \mu^2 - \Delta_x + q) K, \varphi(x) e^{\mu t} \chi(t) \rangle \\
 & \hspace{20pt}= \langle (-\Delta_x + q - (\sigma+i\mu)^2) \psi_q^s(x, \sigma+i\mu, \omega), \varphi(x) (e^{\mu t} \chi)\ebreve(\sigma) \rangle_{\mR^n_x \times \mR_{\sigma}} \\
 & \hspace{20pt}= \langle -q(x) e^{i(\sigma+i\mu)x \cdot \omega},  \varphi(x) (e^{\mu t} \chi)\ebreve(\sigma) \rangle_{\mR^n_x \times \mR_{\sigma}} \\
 & \hspace{20pt}= \langle -q(x) e^{-\mu x \cdot \omega} \delta(t-x \cdot \omega),  \varphi(x) e^{\mu t} \chi(t)  \rangle \\
 & \hspace{20pt}= \langle -q(x) \delta(t-x \cdot \omega),  \varphi(x) \chi(t)  \rangle.
\end{align*}
This proves \eqref{boxqvq_equation}.
\end{proof}

It is now easy to complete the reduction from the scattering amplitude to time domain measurements.

\begin{proof}[Proof of Theorem \ref{thm_frequencydomain_reduction}]
Let $r_0 = \max(-\inf q_1, -\inf q_2, 0 )^{1/2}$. Proposition \ref{prop_resolvent_meromorphic} states that the resolvents $R_{q_j}(\lambda)$ are well defined for $\lambda \in \mC_+ \setminus i(0,r_0]$, and thus for such $\lambda$ one may define 
\begin{align*}
\psi_{q_j}^s(\,\cdot\,,\lambda,\omega) = R_{q_j}(\lambda)(-q_j e^{i\lambda x \cdot \omega}).
\end{align*}
By Proposition \ref{prop_resolvent_meromorphic}, the map $\lambda \mapsto \psi_{q_j}^s(\,\cdot\,,\lambda,\omega)$ extends continuously as a map $\ol{\mC}_+ \setminus i[0,r_0] \to L^2_{\mathrm{loc}}(\mR^n)$ (this is the limiting absorption principle, see e.g.\ \cite[Section 6.2]{Yafaev}). By \cite[Section 6.7]{Yafaev}, for any $\lambda > 0$ the limit satisfies 
\begin{equation} \label{psiq_asymptotics}
\psi_{q_j}^s(r\theta,\lambda,\omega) = e^{i\lambda r} r^{-\frac{n-1}{2}} a_{q_j}(\lambda, \theta, \omega) + o(r^{-\frac{n-1}{2}}), \qquad r \to \infty.
\end{equation}

Assume first  that $a_{q_1}(\lambda,\theta,\omega) = a_{q_2}(\lambda,\theta,\omega)$ for all $\lambda \geq \lambda_0$ and all $\theta$. Together with the fact that $q_1$ and $q_2$ vanish outside $\ol{B}$, this implies that for any fixed $\lambda \geq \lambda_0$, the function $\psi_{q_1}^s - \psi_{q_2}^s$ satisfies 
\begin{align*}
(-\Delta-\lambda^2) ((\psi_{q_1}^s - \psi_{q_2}^s)(\,\cdot\,,\lambda,\omega)) &= 0 \text{ in $\mR^n \setminus \ol{B}$}, \\
(\psi_{q_1}^s - \psi_{q_2}^s)(x,\lambda,\omega) &= o(\abs{x}^{-\frac{n-1}{2}}) \text{ as $\abs{x} \to \infty$.}
\end{align*}
The Rellich uniqueness theorem (see e.g.\ \cite{Hormander_rellich}) implies that $\psi_{q_1}^s - \psi_{q_2}^s$ vanishes outside $\ol{B}$. In particular, for any $\varphi \in C^{\infty}_c(\mR^n \setminus \ol{B})$, the function 
\[
w_{\varphi}(\lambda) = \langle (\psi_{q_1}^s - \psi_{q_2}^s)(\,\cdot\,,\lambda,\omega), \varphi \rangle
\]
satisfies 
\[
w_{\varphi}|_{[\lambda_0,\infty)} = 0.
\]
However, by Proposition \ref{prop_resolvent_meromorphic} the function $\lambda \mapsto w_{\varphi}(\lambda)$ is holomorphic in $\mC_+ \setminus i(0,r_0]$ and has a continuous extension to $\ol{\mC}_+ \setminus i[0,r_0]$. Since it vanishes on $[\lambda_0,\infty)$, one must have $w_{\varphi}(\lambda) \equiv 0$. In particular, for any $\mu > r_0$ one has 
\[
\langle (\psi_{q_1}^s - \psi_{q_2}^s)(x,\sigma+i\mu,\omega), \varphi(x) \rangle_{\mR^n_x} = 0, \qquad \sigma \in \mR.
\]
The relation in Proposition \ref{prop_uq_psiq_relation} then implies that 
\[
\langle u_{q_1}(x,t,\omega) - u_{q_2}(x,t,\omega), \varphi(x) \chi(t) \rangle_{\mR^n_x \times \mR_t} = 0
\]
for all $\varphi \in C^{\infty}_c(\mR^n \setminus \ol{B})$ and $\chi \in C^{\infty}_c(\mR)$. This means that 
\[
u_{q_1}(x,t,\omega) = u_{q_2}(x,t,\omega), \qquad (x,t) \in (\mR^n \setminus \ol{B}) \times \mR.
\]
In particular, one has $u_{q_1}(x,t,\omega) = u_{q_2}(x,t,\omega)$ for $(x, t) \in (S \times \mR) \cap \{ t \geq x \cdot \omega \}$ as required.

Let us now prove the converse. Assume for simplicity that $\omega = e_n$, and assume that $u_{q_1}(x,t,e_n) = u_{q_2}(x,t,e_n)$ for $(x, t) \in (S \times \mR) \cap \{ t \geq z \}$. By Proposition \ref{prop_uq_basic}, the function $\alpha := u_{q_1} - u_{q_2}$ solves 
\begin{align*}
\Box \alpha &= 0 \qquad \text{in $\{ (x,t) \,;\, \abs{x} > 1 \text{ and } t > z\}$}, \\
\alpha(y,z,z) &= -\frac{1}{2} \int_{-\infty}^z (q_1-q_2)(y,s) \,ds \text{ on $\{ \abs{x} > 1\}$}, \\
\alpha &=0 \qquad \text{in $\{ z < t < -1 \}$}.
\end{align*}
Moreover, $\alpha|_{(S \times \mR) \cap \{ t > z \}} = 0$. Thus by Lemma \ref{lemma_energy_estimate_exterior} one also has $\p_{\nu} \alpha|_{(S \times \mR) \cap \{ t > z \}} = 0$. Now the Cauchy data of $\alpha$ vanishes on the lateral boundary of the set $\{ (x,t) \,;\, \abs{x} \geq 1 \text{ and } t \geq z\}$, and Holmgren's uniqueness theorem applied in this set shows that $\alpha$ is identically zero in the relevant domain of dependence. However, by finite speed of propagation the support of $\alpha$ is contained in the same domain of dependence. Thus $\alpha$ is identically zero in $\{ (x,t) \,;\, \abs{x} \geq 1 \text{ and } t \geq z\}$, which implies that 
\[
u_{q_1}(x,t,e_n) = u_{q_2}(x,t,e_n), \qquad (x,t) \in (\mR^n \setminus \ol{B}) \times \mR.
\]
The relation in Proposition \ref{prop_uq_psiq_relation} now gives that for any $\mu > r_0$ and for any $\varphi \in C^{\infty}_c(\mR^n \setminus \ol{B})$, 
\[
\langle (\psi_{q_1}^s - \psi_{q_2}^s)(x,\sigma+i\mu,e_n), \varphi(x) \rangle_{\mR^n_x} = 0, \qquad \sigma \in \mR.
\]
Since by Proposition \ref{prop_resolvent_meromorphic} the function $\lambda \mapsto \langle (\psi_{q_1}^s - \psi_{q_2}^s)(\,\cdot\,,\lambda,e_n), \varphi \rangle$ is holomorphic in $\mC_+ \setminus i(0,r_0]$ and has a continuous extension to $\ol{\mC}_+ \setminus i[0,r_0]$, it follows that 
\[
\langle (\psi_{q_1}^s - \psi_{q_2}^s)(x,\lambda,e_n), \varphi(x) \rangle_{\mR^n_x} = 0, \qquad \lambda > 0.
\]
Thus $\psi_{q_1}^s(\,\cdot\,,\lambda,e_n) - \psi_{q_2}^s(\,\cdot\,,\lambda,e_n)$ vanishes outside $\ol{B}$ for any $\lambda > 0$. By the asymptotics given in \eqref{psiq_asymptotics}, this implies that $a_{q_1}(\lambda,\theta,e_n) = a_{q_2}(\lambda,\theta,e_n)$ for all $\lambda > 0$ and $\theta \in S^{n-1}$ as required.
\end{proof}

\appendix

\section{Carleman estimates for second order PDEs} \label{sec_carleman_estimate}

This exposition of the statement and the derivation of Carleman estimates with boundary terms for second order operators
with real coefficients is based mostly on Chapter 4 of \cite{Ta99} and Chapter VIII of \cite{Ho76}. What is new here is the 
explicit expression for
the boundary terms and perhaps our explanations are not as terse as in \cite{Ta99}.

\subsection{The Carleman estimate}

We use the following notation in this exposition. For complex valued functions $f(x)$ on $\R^n$,
$f_j = \pa_j f = \frac{\pa f}{\pa x_j}$, $\pa f = (\pa_1 f, \cdots, \pa_n f)$,  $D_j f = \frac{1}{i} \pa_j f$ and 
$S = \{ (\xi, \sigma) \in \R^n \times \R : |\xi|^2 + \sigma^2 =1 \}$. Further,
 $\Omega$ will represent a bounded open subset of $\R^n$ with Lipschitz boundary and 
\[
P(x,D) = \sum_{j=1}^n \sum_{k=1}^n a^{jk}(x) D_j D_k + \sum_{j=1}^n b^j(x) D_j + c(x) 
\]
will be a second order operator with $a^{jk}=a^{kj}$ being real valued functions in $C^1(\Omb)$,
and $b^j$, $c$ are bounded complex valued 
functions on $\Omb$. We often drop the summation symbol when it is clear from the context that a summation is involved.
The principal symbol of $P(x,D)$ is the function
\[
p(x,\xi) = a^{jk}(x) \xi_j \xi_k, \qquad x \in \Omb, ~ \xi \in \R^n;
\]
note that the double summation over $j,k$ is implied in the above definition.

For differentiable functions $p(x,\xi)$ and $q(x,\xi)$ on $\Omb \times \R^n$, we define their Poisson bracket as
\[
\{p,q\} = \sum_{j=1}^n \frac{\pa p}{\pa \xi_j} \frac{ \pa q}{\pa x_j} -  \frac{\pa p}{\pa x_j} \frac{ \pa q}{\pa \xi_j}
\]

\begin{Definition} 
Suppose $\phi(x)$ is a real valued smooth function on $\Omb$ with $(\pa \phi)(x) \neq 0$ at each point $x \in \Omb$.
The level surfaces of $\phi$ are said to be pseudoconvex with respect to $P(x,D)$ on $\Omb$ if
\beqn
\{p, \{p, \phi\}\}(x,\xi) >0
\label{eq:pseudo}
\eeqn
for all $ x \in \Omb$ and all non-zero $\xi \in \R^n$ satisfying
\beqn
p(x, \xi)=0, ~~ \{p, \phi\}(x,\xi)=0.
\label{eq:pseudocond}
\eeqn
\end{Definition}
\begin{Definition} 
Suppose $\phi(x)$ is a real valued smooth function on $\Omb$ with $(\pa \phi)(x) \neq 0$ at each point $x \in \Omb$.
The level surfaces of $\phi$ are said to be strongly pseudoconvex with respect
to $P(x,D)$ on $\Omb$ if the level surfaces of $\phi$ are pseudoconvex and
\beqn
\frac{1}{i \sigma} \{ \overline{ p(x, \zeta) }, p(x, \zeta) \} >0
\label{eq:strong}
\eeqn
for all $ x \in \Omb$ and all $\zeta = \xi + i \sigma \pa
\phi(x)$, $\xi \in \R^n$, $\sigma \neq 0$, satisfying
\beqn
p(x, \zeta)=0, ~~ \{ p(x, \zeta), \phi(x) \} = 0.
\label{eq:strongcond}
\eeqn
\end{Definition}
The following proposition (Theorem 1.8 in \cite{Ta99}) is useful in constructing weights for Carleman estimates.
\begin{Proposition}\label{prop:level}
Suppose $\Omega$ a bounded open subset of $\R^n$ with Lipschitz boundary, 
$P(x,D)$ is a second order differential operator  on $\Omb$
with the principal part having real coefficients, and $\phi$ is a real valued smooth function on $\Omb$ with $\pa \phi$ never zero
on $\Omb$. The level surfaces of $\phi$ are strongly pseudoconvex on $\Omb$ iff they are pseudoconvex on $\Omb$.
\end{Proposition}

We prove the Carleman estimates for weights $\phi $ which satisfy the strong pseudoconvexity condition defined below.
\begin{Definition}
Suppose $\phi(x)$ is a real valued smooth function on $\Omb$ with $(\pa \phi)(x) \neq 0$ at each point $x \in \Omb$.
We say that $\phi$ is strongly pseudoconvex on $\Omb$ with respect to $P(x,D)$ if for all $x \in \Omb$ and all 
$\xi \in \R^n$ we have
\beqn
\{p, \{p, \phi\}\}(x,\xi) >0, \qquad \text{when} ~ p(x, \xi)=\{p, \phi\}(x,\xi)=0, ~ \xi \neq 0,
\eeqn
and
\beqn
\frac{1}{i \sigma} \{ \overline{ p(x, \zeta) }, p(x, \zeta) \} >0, \qquad 
\text{when} ~ p(x,\zeta)=0, ~ \zeta=\xi + i\sigma \pa \phi(x), ~\sigma \neq 0.
\eeqn
\end{Definition}
Note that we make a 
distinction between the phrases ``level surfaces of $\phi$ are strongly pseudoconvex'' and ``$\phi$ is strongly pseudoconvex''.
If $\phi$ is strongly pseudoconvex w.r.t $P(x,D)$ on $\Omb$ then the level surfaces of $\phi$ are clearly strongly pseudoconvex 
w.r.t $P(x,D)$ on $\Omb$, but the converse is not true. However, $\phi$ needs to be strongly pseudoconvex for Carleman estimates
to hold. The following proposition (\cite{Ho76}, Theorem 8.6.3) is useful in constructing strongly pseudoconvex weights.
\begin{Proposition}\label{prop:special}
Suppose $\Omega$ a bounded open subset of $\R^n$ with Lipschitz boundary, 
$P(x,D)$ is a second order differential operator  on $\Omb$
with the principal part having real coefficients, and $\psi$ is a real valued function in $C^1(\Omb)$ with 
$\pa \psi$ never zero
on $\Omb$. If the level surfaces of $\psi$ are strongly pseudoconvex with respect to
$P(x,D)$ on $\Omb$, then for large enough real $\lambda$,
$\phi = e^{\lambda \psi}$ is strongly pseudoconvex with respect to $P(x,D)$ on $\Omb$.
\end{Proposition}
It is often easier to construct suitable functions whose level surfaces are 
pseudoconvex, than to directly construct functions which are strongly pseudoconvex. However, Carleman estimates require strongly pseudoconvex 
functions. So one first
constructs a useful function $\psi$ whose level surfaces are pseudoconvex. Then, by Proposition \ref{prop:level},
 the level surfaces of $\psi$ are
strongly pseudoconvex and hence, by Proposition \ref{prop:special}, $\phi= e^{\lambda \psi}$
is strongly pseudoconvex for large enough $\lambda$. Further, $\psi$ and $\phi$ have the same level surfaces.

In verifying pseudoconvexity of level surfaces of $\phi$, it is useful to have explicit expressions for (\ref{eq:pseudo})
and (\ref{eq:strong}). These are available in \cite{Ho76} and one has
\begin{align}
\{p, \{p,\psi \} \} 
& = \psi_{jk} \, \frac{ \pa p}{\pa \xi_j} \, \frac{ \pa p}{\pa \xi_k}
+ \left ( \frac{ \pa p_k}{\pa \xi_j} \, \frac{ \pa p}{\pa \xi_k} - p_k
\frac{ \pa^2 p}{\pa \xi_j \pa \xi_k} \right ) \psi_j
\label{eq:pseudo1}
\\
\frac{1}{i \sigma} \{ \overline{ p(x, \zeta) }, p(x, \zeta) \}
& =
\psi_{jk}(x) \, \frac{ \pa p}{\pa \xi_j}(x, \zeta) \, \overline{
\frac{ \pa p}{\pa \xi_k}(x,\zeta) }
+ \sigma^{-1} \text{Im} \left ( p_k (x,\zeta) \,
\overline{\frac{ \pa p}{\pa \xi_k}(x,\zeta)} \right ).
\label{eq:strong1}
\end{align}

The strong pseudoconvexity of $\phi$ may be expressed as a positive definiteness condition which will be useful when proving Carleman estimates.
\begin{Lemma}\label{lemma:positive}
If $\phi$ is strongly pseudoconvex w.r.t $P(x,D)$ on $\Omb$ then there is a constant $c>0$ such that for
$\zeta = \xi + i \sigma \pa \phi$ we have
\beqn
\frac{1}{i \sigma} \{ \overline{ p(x, \zeta) }, p(x, \zeta) \}
\geq c , \qquad \text{for } (x, \xi, \sigma) \in \Omb \times S ~~ \text{with} ~ 
p(x, \xi)-\sigma^2 p(x,\pa \phi) = \{ p, \phi \}(x,\xi) =0.
\label{eq:uni-pseudo}
\eeqn
Here, the value of the LHS, when $\sigma=0$, is to be understood in the sense of a limit as $\sigma \to 0$.
\end{Lemma}
\begin{proof}
We have
\begin{align*}
p(x,\zeta)  & = a^{jk} (\xi_j + i \sigma \phi_j) (\xi_k + i \sigma \phi_k)
\\
& = a^{jk} \xi_j \xi_k - \sigma^2 a^{jk} \phi_j \phi_k +  i \sigma a^{jk} \xi_k \phi_j + i \sigma a^{jk} \xi_j \phi_k
\\
& = p(x,\xi) - \sigma^2 p(x, \pa \phi) + i \sigma \frac{\pa p}{\pa \xi_j} \phi_j
\\
& = A(x, \xi, \sigma) + i \sigma B(x,\xi)
\end{align*}
where
\[
A(x, \xi, \sigma) = p(x,\xi) - \sigma^2 p(x, \pa \phi), \qquad B(x,\xi) = \{p, \phi\}(x,\xi)
\]
are real valued.
Hence, for $\sigma \neq 0$, using $\{A,A\}=0$, $\{B,B\}=0$ and $\{B,A\} = - \{A,B\}$, we have
\begin{align*}
\frac{1}{2i \sigma} \{ \overline{p(x,\zeta)}, p(x,\zeta) \} 
& =\frac{1}{2i \sigma} \{ A(x,\xi,\sigma) - i \sigma B(x,\xi), A(x, \xi, \sigma) + i \sigma B(x,\xi) \}
\\
& = \{A,B \}  = \{ p, \{p, \phi\}\}(x,\xi) - \sigma^2 \{ p(x, \pa \phi), \{p, \phi\}\}(x,\xi)
\\
& = \{ p, \{p, \phi\}\}(x,\xi) + \sigma^2 \{  \{p, \phi\}, p(x, \pa \phi)\}(x,\xi)
\\
& = \{ p, \{p, \phi\}\}(x,\xi)  + \sigma^2 \{ p, \{p, \phi\}\}(x, \pa \phi),
\end{align*}
where the last step follows from the relation 
\beqn
\{ p, \{p, \phi\}\}(x,\pa \phi) = \{  \{p, \phi\}, p(x, \pa \phi)\}(x,\xi)
\label{eq:temp1c}
\eeqn
which is verified at the end of this proof. Hence
\[
\lim_{\sigma \to 0} \, \frac{1}{2 i \sigma}  \{ \overline{p(x,\zeta)}, p(x,\zeta) \}  = \{ p, \{p, \phi\}\}(x,\xi).
\]
So if we define $\frac{1}{2 i \sigma}  \{ \overline{p(x,\zeta)}, p(x,\zeta) \}$ to be  $\{ p, \{p, \phi\}\}(x,\xi)$ when $\sigma =0$ 
then 
$\frac{1}{2 i \sigma}  \{ \overline{p(x,\zeta)}, p(x,\zeta) \}$ is a continuous real valued function on the compact set
$\Omb \times S$. 
Now the definition of strong pseduoconvexity guarantees that $\frac{1}{2 i \sigma}  \{ \overline{p(x,\zeta)}, p(x,\zeta) \}$ is positive 
on the set
\[
\{ (x,\xi, \sigma) \in \Omb \times S :  p(x,\xi) - \sigma^2 p(x, \pa \phi)=0 = \{ p, \phi \}(x,\xi) \} 
\]
provided $\sigma \neq 0$. When $\sigma=0$, the points on this set lie in 
\[
\{ (x,\xi) \in \Omb \times \R^n : ~ \xi \neq 0, ~ p(x, \xi)=0, ~ \{p, \phi \}(x, \xi)=0 \}
\]
and $\{ p, \{p, \phi \} \}$ is positive on this set by the definition of strong pseudoconvexity. Hence the Lemma follows by continuity 
and compactness.

It remains to verify (\ref{eq:temp1c}) which we do now using Euler's identity for homogeneous functions and the 
fact that 
$\frac{ \pa p}{\pa \xi_j} (x,\xi)$ is a homogeneous of degree $1$ in $\xi$ and $p_j(x,\xi)$ is homogeneous of degree
$2$ in $\xi$. We have
\begin{align*}
 \{ \{p(x,\xi), \phi\}, p(x, \pa \phi) \}(x,\xi) &= \frac{\pa}{\pa \xi_j} \left (  \frac{\pa p}{\pa \xi_k} (x, \xi)
\; \phi_k(x) \right ) \left ( p_j(x, \pa \phi) +  \frac{\pa p}{\pa \xi_k}(x, \pa \phi) \phi_{jk}(x) \right )
\\
&= \left (  \frac{\pa^2 p}{\pa \xi_j \xi_k} (x, \xi)
\; \phi_k(x) \right ) \left ( p_j(x, \pa \phi) +  \frac{\pa p}{\pa \xi_k}(x, \pa \phi) \phi_{jk}(x) \right )
\\
&= \frac{\pa p}{\pa \xi_j} (x, \pa \phi) \left ( p_j(x, \pa \phi) +  \frac{\pa p}{\pa \xi_k}(x, \pa \phi) \phi_{jk}(x) \right )
\end{align*}
since $\frac{\pa^2 p}{\pa \xi_j \xi_k} (x, \xi) \phi_k(x) = \frac{\pa^2 p}{\pa \xi_j \xi_k} (x, \xi) \phi_k(x)|_{\xi=\pa \phi} = \frac{\pa p}{\pa \xi_j} (x, \pa \phi)$, and
\begin{align*}
 \{p,\{p,\phi\}\}(x, \pa \phi) &= \frac{\pa p}{\pa \xi_j} (x, \pa \phi) \{p, \phi\}_j(x, \pa \phi)
- p_j(x) \left ( \frac{\pa \{p, \phi\}}{\pa \xi_j}   \right )(x, \pa \phi)
\\
&= \frac{\pa p}{\pa \xi_j} (x, \pa \phi)
\left ( \frac{\pa p_j}{\pa \xi_k} (x, \pa \phi) \phi_k
+ \frac{\pa p}{\pa \xi_k} (x, \pa \phi) \phi_{jk}
\right )
- p_j \frac{\pa^2 p}{\pa \xi_k \xi_j} (x, \pa \phi) \phi_k
\\
&=\frac{\pa p}{\pa \xi_j} (x, \pa \phi)
\left ( 2p_j (x, \pa \phi)
+ \frac{\pa p}{\pa \xi_k} (x, \pa \phi) \phi_{jk}
\right )
- p_j \frac{\pa p}{\pa \xi_j} (x, \pa \phi)
\\
&=\frac{\pa p}{\pa \xi_j} (x, \pa \phi)
\left ( p_j (x, \pa \phi)
+ \frac{\pa p}{\pa \xi_k} (x, \pa \phi) \phi_{jk}
\right ). \qedhere
\end{align*}
\end{proof}

%%%%%%%%%%%%%%%%%%%%%%%%%%%%%%%%%%%%%

Here is the main result about Carleman estimates with boundary terms.
%%%%%%%%%
%
\begin{Theorem}\label{thm:carleman}
Suppose $\Omega$ is a bounded open set in $\R^n$, $n \geq 2$, with a Lipschitz boundary, and $P(x,D)$ is a second order differential operator on 
$\Omb$ with bounded coefficients whose principal symbol $p(x,\xi)$ has real $C^1$ coefficients. If $\phi$ is a smooth function on 
$\Omb$ with $\pa \phi$ never zero in $\Omb$ and $\phi$
is strongly pseudoconvex with respect to $P(x,D)$ on $\Omb$, then for large
enough $\sigma$ and for all real valued $u \in C^2(\Omb)$ one has 
\beqn
\sigma \int_\Omega e^{2 \sigma \phi} (|\pa u|^2 + \sigma^2 u^2)
+\sigma \int_{\pa \Omega} \nu^j E^j
\cleq  \int_\Omega e^{2 \sigma \phi} |P u|^2,
\label{eq:ucarl}
\eeqn
with the constant independent of $\sigma$ and $u$. Here $\nu=(\nu^1, \cdots, \nu^n)$ is the outward unit normal to $\pa \Omega$,
\[
E^j := A(x, \pa v, \sigma v) \, \frac{ \pa B}{\pa \xi_j}(x) 
- \frac{ \pa A}{\pa \xi_j} (x, \pa v,\sigma v) \, ( B(x, \pa v) + g(x) v),
\]
$v = e^{\sigma \phi} u$, $g$  some real valued function independent of $\lambda, \sigma, u,$ and
\beqn
A(x, \xi, \sigma) := p(x,\xi) - \sigma^2 p(x, \pa \phi),
~~
B(x, \xi) := \{p, \phi\}(x, \xi).
\label{eq:ABform}
\eeqn
\end{Theorem}
\begin{Remark}
It is not difficult to see that the expressions for $E^j$ and (\ref{eq:ucarl}) imply that 
\[
\sigma \int_\Omega e^{2 \sigma \phi} (|\pa u|^2 + \sigma^2 u^2)
\cleq  \int_\Omega e^{2 \sigma \phi} |P u|^2
+ \sigma \int_{\pa \Omega} e^{2 \sigma \phi} ( |\pa u|^2 + \sigma^2 u^2),
\]
for all $u \in C^2(\Omb)$.
\end{Remark}
%
%%%%%%

%\subsection{Derivation of the Carleman estimate}

%Here we give a proof of Theorem \ref{thm:carleman}.
\begin{proof}
Since the statement of Theorem \ref{thm:carleman} is not affected by a first order perturbation to $P$ we may assume that
$b_j=0$, $c=0$. The Carleman estimate follows quickly from an algebraic inequality derived with the help of Lemma \ref{lemma:positive}. Below 
\[
A(x,\xi,\sigma) = a^{jk} \xi_j \xi_k - \sigma^2 a^{jk} \phi_j \phi_k,
\qquad
B(x,\xi) = \{p(x,\xi),\phi(x)\}
\]
so $A(x,\xi,\sigma)$ is a quadratic form in $(\xi, \sigma)$ and $B(x,\xi)$ is a linear form in $\xi$. Hence
\begin{gather*}
A(x,D, \sigma) = a^{jk} D_j D_k - \sigma^2 a^{jk} \phi_j \phi_k,
\quad 
A(x,\pa v, \sigma v) = a^{jk} v_j v_k - \sigma^2 v^2 a^{jk} \phi_j \phi_k
\\
B(x,D) = \{ p, \phi \}(x,D), \quad B(x,\pa v) = \{ p, \phi \}(x, \pa v).
\end{gather*}
For convenience, sometimes we abbreviate $P(x,D) u(x)$ to $Pu$, $A(x,D,\sigma) v(x)$ to $Av$ and $B(x,D) v(x)$ to $Bv$.

Define $v := e^{\sigma \phi} u$; we show there is a smooth
function $g(x)$, independent of $u$ and $\sigma$,
 so that for large enough $\sigma$
\begin{align}
e^{2 \sigma \phi} |P u|^2
\cgeq \sigma ( |\partial v|^2 + \sigma^2 |v|^2 ) + \sigma
\partial_j E^j, \qquad \text{on} ~ \Omb,
\label{eq:vmain}
\end{align}
with the constant independent of $u,\sigma,x$ and each $E^j$ is a quadratic form in $(\pa v, \sigma v)$ defined in the statement
of Theorem \ref{thm:carleman}. Now $v= e^{\sigma \phi}u$ implies $u = e^{- \sigma \phi} v$ so
$ e^{\sigma \phi} \partial u = \partial v - \sigma \partial \phi \, v$ and
$\partial v = e^{ \sigma \phi} (\partial u + \sigma \partial \phi \, u)$.
Hence
\[
e^{2 \sigma \phi} (|\partial u|^2 + \sigma^2 u^2)
\cleq |\partial v|^2 + \sigma^2 |v|^2
\cleq
e^{2 \sigma \phi} (|\partial u|^2 + \sigma^2 u^2)
\]
with the constant independent of $\sigma, u$ and $x \in \Omb$. Applying this
to (\ref{eq:vmain}) we recover (\ref{eq:ucarl}); so it remains to prove
(\ref{eq:vmain}).

Since $u = e^{-\sigma \phi}v$ we have
$
e^{\sigma \phi} D_j u   =
e^{\sigma \phi} D_j( e^{-\sigma \phi} v ) =
(D_j  +i \sigma  \phi_j)v
$
hence
\begin{align*}
e^{\sigma \phi} p(x,D) u &= p(x, D + i \sigma \pa \phi)v.
\end{align*}
Now
\begin{align*}
p(x, D + i \sigma \pa \phi) &= a^{jk} (D_j + i \sigma \phi_j) (D_k + i \sigma \phi_k)
\\
&= a^{jk}(D_j D_k - \sigma^2 \phi_j \phi_k) + 2 i \sigma a^{jk} \phi_j D_k + \sigma a^{jk} \phi_{jk}
\\
&=A(x,D,\sigma) + i \sigma B(x,D) + \sigma r(x)
\end{align*}
for the known bounded function $r(x) := a^{jk} \phi_{jk}$. Hence, for any real valued  function $g(x) \in C^1(\Omb)$
\begin{align}
e^{2 \sigma \phi} | P u|^2 &= |A v + i \sigma Bv + \sigma rv|^2
= |(A v + i \sigma Bv + \sigma g v) + \sigma (r-g)v|^2
\nn
\\
&\cgeq  | A v + i \sigma Bv + \sigma gv|^2 - 
c\sigma^2  |v|^2 
\nn
\\
&\geq  |Av|^2 + \sigma^2 |Bv|^2 -  i \sigma ( A v \overline{Bv} - \overline{Av} Bv) 
 + 2 \sigma Av \, g v - 2 \sigma^2 gv \,\im(Bv) - c\sigma^2 |v|^2
\nn
\\
&\cgeq \sigma^2 |Bv|^2 -  i \sigma ( A v \overline{Bv} - \overline{Av} Bv) + 2 \sigma Av \, g v
-c \sigma |Bv| \, \sigma |v| -c \sigma^2 |v|^2
\nn
\\
&\cgeq  \sigma^2 |Bv|^2 + 2i \sigma Av \, Bv 
+ 2 \sigma Av \, g v -  c\sigma^2 |v|^2
\label{eq:temp1}
\end{align}
because $Av$ is real and $Bv$ is purely imaginary. Here the constant $c$ may change from line to line and $c$ and the
constant in the inequality depends only on $g, \phi$ and $a^{jk}$. 

Next we express $\sigma^2 |Bv|^2 + 2i \sigma Av \, Bv + 2 \sigma Av \, g v$ as the sum of a divergence of a vector field and a quadratic form in $(\pa v, \sigma v)$ closely tied to the pseudoconvexity condition; see section 8.2 of \cite{Ho76} for a more 
general version of these calculations. 

We first work with $2i Av \, Bv$; $A(x,D, \sigma)v$ is a sum of terms of the form $a(x) D_j D_k v$ and $\sigma^2a(x) v$, and 
$B(x,D)v$ is a sum of terms of the form $b(x)D_m v$ with $a,b,v$ real valued functions. 
If $Av =\sigma^2 a(x) v$ and $Bv = b(x)D_m v$ then
\begin{align}
2i Av \,  Bv 
&=2 \sigma^2 ab v_m v =  \sigma^2 ab (v^2)_m = \sigma^2 (ab v^2)_m - \sigma^2 (ab)_m v^2
\nn
\\
& = - a_m (\sigma v)^2 \, b - \sigma^2 a v^2 \, b_m   + \sigma^2 (ab v^2)_m
\nn
\\
& = \{A, B\}(x, \pa v, \sigma v) - A(x,\pa v,\sigma v) \, b_m + \sum_{l=1}^n  \frac{ \pa }{\pa x_l}
\left ( A(x,\pa v,\sigma v) \,  \frac{ \pa B}{\pa \xi_l}(x) \, \right )
\nn
\\
& =  \{A, B\}(x, \pa v, \sigma v) - A(x,\pa v,\sigma v) 
\sum_{s=1}^n \frac{ \pa^2 B }{\pa \xi_s \, \pa x_s}(x)  
\nn
\\
& \hspace{1in} + \sum_{l=1}^n  \frac{ \pa }{\pa x_l}
\left ( A(x, \pa v, \sigma v) \,  \frac{ \pa B}{\pa \xi_l}(x)  \right ).
\label{eq:ABC}
\end{align}
If $Av = a(x) D_j D_k v$ and $B(x,D)v= b(x)D_m v$ then
\begin{align}
2i Av \, & Bv   = -2ab v_{jk}v_m
=-ab \left ( (v_k v_m)_j + (v_j v_m)_k - (v_j v_k)_m \right )
\nn
\\
&=  (ab)_j v_k v_m + (ab)_k v_j v_m - (ab)_m v_{j}  v_k
-  (ab v_k v_m)_j - (ab v_j v_m)_k + (ab v_{j} v_k)_m
\nn
\\
&=(ab)_j v_k v_m + (ab)_k v_j v_m - (ab)_m v_{j}  v_k 
\nn
\\
& ~~~
+ \sum_l  \frac{ \pa }{\pa x_l} \left ( - \frac{ \pa A}{\pa \xi_l} (x,\pa v,\sigma v) \, B(x, \pa v) + 
A(x, \pa v, \sigma v) \, \frac{ \pa B}{\pa \xi_l}(x)
 \right ).
\label{eq:A1B1}
\end{align}
Now
\begin{align*}
(ab)_j v_k v_m & + (ab)_k v_j v_m - (ab)_m v_{j}  v_k
\\
&=\left ( a v_k \, b_j v_m + a v_j \, b_k v_m - a_m v_{j} v_k \, b \right )
+ \left ( a_j v_k \, b v_m + a_k v_j \, b v_m -  a v_{j} v_k \, b_m  \right )
\\
& = \{ A, B \}(x, \pa v, \sigma v) + M(x, \pa v) B(x,\pa v) -  A(x,\pa v, \sigma v) b_m
\\
&= \{ A, B \}(x, \pa v, \sigma v) + M(x, \pa v) B(x, \pa v) - A(x, \pa v, \sigma v)
 \sum_{s=1}^n \frac{ \pa^2 B }{\pa \xi_s \, \pa x_s}(x)
\end{align*}
where $M(x,\xi) = a_j \xi_k + a_k \xi_j $ is homogeneous and linear of degree $1$ in $\xi$ and
 is independent of $B(x,\xi)$. Hence using (\ref{eq:A1B1}) we have
\begin{align}
2i Av \,  Bv  
&=\{ A, B \}(x, \pa v, \sigma v) + M(x, \pa v) B(x,\pa v) -  
 A(x, \pa v, \sigma v) \sum_{s=1}^n \frac{ \pa^2 B }{\pa \xi_s \, \pa x_s}(x)
\nn
\\
& ~~~
+ \sum_{l=1}^n 
 \frac{ \pa }{\pa x_l}
\left ( 
A(x, \pa v, \sigma v) \, \frac{ \pa B}{\pa \xi_l}(x) 
- \frac{ \pa A}{\pa \xi_l} (x,\pa v,\sigma v) \, B(x, \pa v) 
\right ).
\label{eq:A1B1final}
\end{align}
If $Av = \sigma^2 a(x) v$ then one can see that the last term in (\ref{eq:ABC})
is the same as the last term in (\ref{eq:A1B1final}) because in this case $\frac{ \pa A}{\pa \xi_l}=0$. Hence, since
(\ref{eq:A1B1final})
is bilinear in $A$ and $B$,
we may conclude that for the $A,B$ given by (\ref{eq:ABform})  and for $M$ given by
\[
M(x,\xi) = \sum_{j,k} ( (a^{jk})_j \xi_k + (a^{jk})_k \xi_j) = 2 \sum_{j,k} (a^{jk})_j \xi_k,
\]
one has
\begin{align}
2i Av \, Bv
&= \{ A, B \}(x, \pa v, \sigma v) + M(x,\pa v) B(x,\pa v) -  
A(x, \pa v, \sigma v) \sum_{s=1}^n \frac{ \pa^2 B }{\pa \xi_s \, \pa x_s}(x) 
+ \pa_l F^l
\nn
\\
& \geq \{ A, B \}(x, \pa v, \sigma v) - 
A(x, \pa v, \sigma v) \sum_{s=1}^n \frac{ \pa^2 B }{\pa \xi_s \, \pa x_s}(x) 
+ \pa_l F^l 
\nn
\\
& \hspace{1in} -  c_1 \sqrt{\sigma} \, |B(x,\pa v)|^2 -  \frac{c_2}{\sqrt{\sigma}} |\pa v|^2
\label{eq:algeb}
\end{align}
where 
\[
F^l := A(x, \pa v, \sigma v) \, \frac{ \pa B}{\pa \xi_l}(x)  
- \frac{ \pa A}{\pa \xi_l} (x,\pa v,\sigma v) \, B(x, \pa v).
\]

Now we examine the term $2 Av \,g v$ in (\ref{eq:temp1}). If $Av = \sigma^2 a(x) v$ then
\begin{align}
2 Av \, g v = 2 \sigma^2 a g v^2 = 2A(x,\pa v, \sigma v) g(x).
\label{eq:A0g0}
\end{align}
If $Av=a(x) D_j D_k v$ then
\begin{align}
2 Av \, g v & = -2 a v_{jk} g v 
=  -a g v_{jk} v - a g v_{jk} v
\nn
\\
&=  2a g v_j v_k - (a g v_j v)_k - (a g v_k v)_j
+ (ag)_k v_j v + (ag)_j v_k v
\nn
\\
& = 2 A(x,\pa v, \sigma v) g(x) + N(x, \pa v) v -
\sum_l \frac{ \pa}{\pa x_l}
\left ( \frac{ \pa A}{\pa \xi_l} (x, \pa v, \sigma v) \, g(x) v \right )
\label{eq:A1g0}
\end{align}
where $N(x,\xi) = (ag)_k \xi_j + (ag)_j \xi_k$ is linear in $\xi$.
Note that (\ref{eq:A1g0}) is valid even in the (\ref{eq:A0g0}) case with $N \equiv 0$. Hence using linearity of (\ref{eq:A1g0}) in $A$, for the
$A(x,D, \sigma)v$ given by (\ref{eq:ABform}) we have
\begin{align}
2 A(x,D,\sigma)v \, g(x) v & \geq 2 A(x,\pa v, \sigma v) g(x) -  \frac{ \pa}{\pa x_l} \left ( \frac{ \pa A}{\pa \xi_l} (x, \pa v, \sigma v) \, g v \right )
-  c_1 \sqrt{\sigma} \, |v|^2 -  \frac{c_2}{\sqrt{\sigma}} |\pa v|^2.
\label{eq:Afinal}
\end{align}
So using (\ref{eq:algeb}) and (\ref{eq:Afinal}) in (\ref{eq:temp1}), for large enough $\sigma$ (determined by $\phi$, $a^{jk}$ and
$g$), and using that $\sigma^2 |B(x,\pa v)|^2 \geq \sigma d |B(x,\pa v)|^2$ when $\sigma \geq d$, we obtain
\begin{align}
e^{2 \sigma \phi} | P u|^2
\cgeq &
\ \sigma \{A,B\}(x,\pa v, \sigma v) + \sigma d |B(x,\pa v)|^2 + \sigma h(x) A(x,\pa v, \sigma v) + \sigma \pa_l E^l 
\nn
\\
& \hspace{1in} - c_1 \sqrt{\sigma} \, |\pa v|^2 - c_2 \sigma^2 v^2
\label{eq:temp2}
\end{align}
where
\beqn
h(x) := 2 g(x) - \sum_{s=1}^n \frac{ \pa^2 B }{\pa \xi_s \, \pa x_s}(x) 
\label{eq:hdef}
\eeqn
and
\beqn
E^l := A(x, \pa v, \sigma v) \, \frac{ \pa B}{\pa \xi_l}(x) 
- \frac{ \pa A}{\pa \xi_l} (x,\pa v,\sigma v) \, ( B(x, \pa v) + g(x) v).
\label{eq:Eldef}
\eeqn
The quantity $\{A,B\}(x,\pa v, \sigma v) + d |B(x,\pa v)|^2 + h(x) A(x,\pa v, \sigma v)$ in 
(\ref{eq:temp2}) is a quadratic form in the vector
$(\partial v, \sigma v)$. If we can find a constant $d>0$
and a smooth function $h(x)$ on $\Omb$ so that
\beqn
\{A,B\}(x,\xi,\sigma) + d B(x,\xi)^2 + h(x) A(x,\xi,\sigma) >0, \qquad \text{for } (x,\xi,\sigma) \in \Omb \times S
\label{eq:form}
\eeqn
then from (\ref{eq:temp2}), for large enough $\sigma$,
\begin{align*}
e^{2 \sigma \phi} |P u|^2
& \cgeq \sigma (|\partial v|^2 + \sigma^2 |v|^2) + \sigma \partial_j  E^j
- \sqrt{\sigma} \, |\pa v|^2 -  \sigma^2 |v|^2
\nn \\
& \cgeq \sigma ( |\partial v|^2 + \sigma^2 |v|^2 ) + \sigma
\partial_j E^j,
\end{align*}
proving (\ref{eq:vmain}). Here $g$ is determined by \eqref{eq:hdef} and $h$. So it remains to prove
(\ref{eq:form}).

 For $\zeta=\xi+i \sigma \pa \phi$ we have
\begin{align*}
p(x,\xi+i \sigma \pa \phi) &=A(x,\xi,\sigma) + i \sigma B(x,\xi)
\\
 \frac{1}{i \sigma} \{ \overline{p(x,\zeta)}, p(x,\zeta) \}
 & = \frac{1}{i \sigma} \{ A-i\sigma B, A + i \sigma B \}(x,\xi)
 = 2 \{A(x,\xi,\sigma), B(x,\xi) \},
 \end{align*}
 so, noting that $A(x,\xi,\sigma), B(x,\xi)$ are real valued and homogeneous in $(\xi, \sigma)$, from Lemma 
 \ref{lemma:positive} we have
\beqn
\{A, B \}(x,\xi, \sigma) >0, \qquad \text{for } (x,\xi, \sigma) \in \Omb \times S ~~ \text{with} ~A(x,\xi, \sigma)=0, ~
B(x, \xi)=0.
\label{eq:ABP}
\eeqn
Hence\footnote{
There is an $\epsilon>0$ so that $\{A,B\}(x, \xi, \sigma)$ is positive on
$ \{ (x, \xi, \sigma) \in \Omb \times S : A(x,\xi,\sigma)=0, |B(x,\xi)|^2 \leq \epsilon \}$. Otherwise, there would be
a convergent sequence $(x_k, \xi_k, \sigma_k)$ in $\Omb \times S$ for which 
$A(x_k, \xi_k, \sigma_k)=0$, 
$|B(x_k, \xi_k)|^2 \to 0$ and $\{A,B\}(x_k, \xi_k, \sigma_k) \leq 0$; then 
taking limits we would violate (\ref{eq:ABP}). So assume there is such a
positive $\epsilon$; then choose $d$ large enough so that $d \epsilon$
exceeds the maximum of $|\{A,B\}(x,\xi,\sigma))|$ over
$\{ (x, \xi, \sigma) \in \Omb \times S : A(x, \xi, \sigma)=0 \}$.}
we can find a $d>0$ so that
\beqn
\{A,B\}(x,\xi,\sigma) + d |B(x,\xi)|^2 >0, \qquad \text{for }
(x, \xi, \sigma) \in \Omb \times S ~ \text{with} ~ A(x,\xi,\sigma)=0.
\label{eq:Apos}
\eeqn
Now fix an $x \in \Omb$ and define the following quadratic forms in $(\xi, \sigma$)
\begin{align*}
q(\xi, \sigma) & := \{A,B\}(x,\xi,\sigma) + d |B(x,\xi)|^2, \\
q_\lambda(\xi,\sigma) &:= q(\xi,\sigma) + \lambda A(x,\xi,\sigma).
\end{align*}
If we can find some constant $\lambda$ so that
$q_{\lambda}(\xi,\sigma) >0$
for all $(\xi, \sigma) \in S$, then the same $\lambda$ will work
in a neighborhood (in $\Omb$) of this $x$. Hence, using a partition of unity
argument, we can construct quadruples $(U_j, V_j, \chi_j, \lambda_j)$, $j=1, \cdots, m$, with 
%
%\vspace{-0.3in}
\begin{itemize}
\item $U_j, V_j$ open subsets of $\R^n$,  $\overline{U_j} \subset V_j$ and
$ \Omb \subset \cup_{j=1}^m U_j$;
\item $\chi_j \in C_c^\infty(V_j)$,  $\chi_j$ nonnegative, $\chi_j >0$ on $U_j$ and $\sum_{j=1}^m \chi_j = 1$ 
on $\Omb$;
\item $\lambda_j \in \R$ and $q_{\lambda_j}(\xi, \sigma)>0$ for all $(x, \xi, \sigma) \in (\Omb \cap V_j) \times S$. 
\end{itemize}
%
%\vspace{-0.2in}
%
Hence, if $h= \sum_{j=1}^m \lambda_j \chi_j$
then (\ref{eq:form}) holds for all $(x, \xi, \sigma)  \in \Omb \times S$ because
\begin{align*}
\{A,B\}(x,\xi,\sigma)  & + d B(x,\xi)^2 + h(x) A(x,\xi,\sigma)
\\
& = \{A,B\}(x,\xi,\sigma) + d B(x,\xi)^2 + A(x,\xi,\sigma) \sum_{j=1}^m \lambda_j \chi_j(x) 
\\
& = \sum_{j=1}^m \chi_j(x) \left ( \{A,B\}(x,\xi,\sigma) + d B(x,\xi)^2 + \lambda_j A(x,\xi,\sigma) \right ).
\end{align*}
So we take $g$ to be the function which
satisfies (\ref{eq:hdef}). It remains to show that (\ref{eq:Apos}) implies for any fixed $x \in \Omb$
there is a $\lambda \in \R$ with
$q_\lambda(\xi, \sigma)>0$ for all $(\xi, \sigma) \in S$.

Fix an $x \in \Omb$. Let $\Zl$ be the zero set of the quadratic
form $q_\lambda(\xi,\sigma)$ in $\R^{n+1} \setminus \{0\}$ - then $\Zl$ is a
collection of lines in $(\xi,\sigma)$ space. We claim that $\Zl$ (or
the zero set of any quadratic form) is projectively connected,
that is, there is a continuously varying family of lines
in $Z_\lambda$
connecting any two lines in $\Zl$. Without loss of generality we assume the quadratic form is 
generated by a diagonal
matrix with $l$ ones, $m$ minus ones, and $k$ zeros - we
prove the claim by induction on $l$. If $l=0$ or $m=0$ then it is trivial
so assume $l \geq 1$, $m \geq 1$. If $l=1$ then the zero set is a
cone times $\mR^k$ and hence projectively connected (if $l=m=1$ we need to use that $k \geq 1$, which 
follows since $n \geq 2$). If $l \geq 2$ 
and the line through the origin and $(p,q,r) \neq 0$ is in the zero set with 
$p \in \mR^l,  q \in \mR^m, r \in \mR^k$ then $|p|^2 = |q|^2$. We can find a $p' \in
\mR^{l-1}$ so that $|p'|^2 = |p|^2=|q|^2$; also we can connect $p$ to
$(p',0)$ by a curve on a ball of radius $|p|$. Hence the zero set
of the quadratic form is projectively connected to the zero set of a quadratic form with
signature $l-1,m, k$ and this zero set is projectively connected
by the induction hypothesis.

Now $q>0$ on $S \cap \{A=0\}$ by (\ref{eq:Apos}), hence $q>0$ on $S \cap \{ |A| \leq \epsilon\}$
for some $\ep>0$. Hence
%\vspace{-0.2in}
\begin{itemize}
\item $q_{\lambda}=q+\lambda A>0$ on $S \cap \{A>0\}$ if  $\lambda > \epsilon^{-1}\max_S |q|$,
\item  $q_\lambda=q+\lambda A >0$ on $S \cap \{A<0\}$ if $\lambda< - \ep^{-1} \max_S |q|$,
\end{itemize}
%\vspace{-0.2in}
so
\beqn
\text{
 $\Zl \cap S$ is contained in $A<0$ for $\lambda \gg 0$ and
$\Zl \cap S$ is contained in $A>0$ for $\lambda \ll 0$.
}
\label{eq:Z}
\eeqn
We claim that this implies $\Zl \cap S$ is
empty for some $\lambda$, that is for some $\lambda$, $q_\lambda$
is never zero on $S$ and hence has the same sign at every point on
$S$. But $q_\lambda>0$ on $A=0$ so $q_\lambda>0$ on $S$ which
would prove our claim. It remains to show that (\ref{eq:Apos}), (\ref{eq:Z}) imply $\Zl \cap S$ is empty for
some $\lambda$.

We argue by contradiction and suppose that $\Zl \cap S \neq \emptyset$ for all $\lambda \in \R$.
From (\ref{eq:Apos}) and the projective connectedness of $\Zl$,  
$\Zl \cap S$ is contained either in the set $A>0$ or the set $A<0$. Thus $\mR = \Lambda_+ \cup \Lambda_-$, where the sets $\Lambda_+$ and $\Lambda_-$ are defined as 
\[
\Lambda_+ := \{ \lambda \in \R : \Zl \cap S \subset \{A>0\} \},
\qquad 
\Lambda_- := \{ \lambda \in \R : \Zl \cap S \subset \{A<0\} \}.
\]
The sets $\Lambda_+$ and $\Lambda_-$ are non-empty because of
(\ref{eq:Z}) and disjoint since $\Zl \cap S \neq \emptyset$ for all $\lambda$. They are also closed: if there is a sequence 
$\lambda_k\rightarrow \lambda^*$ with $Z_{\lambda_k} \cap S$
contained in $A>0$ for all $k$,  there is a convergent sequence
$(\xi_k,\sigma_k) \to (\xi^*,\sigma^*)$ in $S$ with $A(\xi_k,\sigma_k)>0$ and
$q_{\lambda_k}(\xi_k, \sigma_k)=0$. Taking the limit we have
$q_{\lambda^*}(\xi^*, \sigma^*)=0$ and $A(\xi^*,\sigma^*) \geq 0$, which by \eqref{eq:Apos} implies
$q_{\lambda^*}(\xi^*, \sigma^*)=0$ and $A(\xi^*, \sigma^*) > 0$ so  $Z_{\lambda^*}\cap S$ is contained in 
$A>0$. Hence $\Lambda_+$ is closed and by a similar argument $\Lambda_-$ is closed. But now one has $\mR = \Lambda_+ \cup \Lambda_-$ where $\Lambda_+$ and $\Lambda_-$ are nonempty, disjoint and closed sets. This contradicts the connectedness of $\mR$.
\end{proof}

%\hfill{\bf QED}

%%%%%%%%%%%%%%%%%%%%%%%%%%
\subsection{Boundary terms for the wave operator} \label{subsection_boundary_terms_wave}

We determine the boundary terms in Theorem \ref{thm:carleman} for the wave operator $\Box$.
Here the independent variables are $(x,t) \in \R^n \times \R$,  $\Box = \pa_t^2 - \Delta_x$ and the Carleman weight
function is $\phi(x,t)$. So the principal symbol of $\Box$ is
\[
p(\xi,\tau)  = -\tau^2 + \xi \cdot \xi.
\]

%%%%%%%%

\noindent
{\bf Expressions for $A$, $B$.}

\vspace{10pt}

\noindent Now, if $ \zeta = (\xi, \tau) + i \sigma (\phi_x, \phi_t)$ then
\begin{align*}
p(\zeta) & = - (\tau + i \sigma \phi_t)^2 + (\xi+ i \sigma \phi_x) \cdot ( \xi+ i \sigma \phi_x)
\\
& = ( |\xi|^2 - \tau^2) - \sigma^2( |\phi_{x}|^2 - \phi_t^2) + 2 i \sigma ( \xi \cdot \phi_x - \tau \phi_t),
\end{align*}
hence
\[
A(x, t, \xi, \tau, \sigma) =  ( |\xi|^2 - \tau^2) - \sigma^2( |\phi_{x}|^2 - \phi_t^2) ,
\qquad
B(x,t,\xi, \tau) = 2 ( \xi \cdot \phi_x - \tau \phi_t).
\]

\vspace{5pt}

%%%%%%%%%%%%%%%%%%%%%%%%%%

\noindent
{\bf Expressions for the boundary terms $E^j$ for $\Box$.}

\vspace{10pt}

\noindent For $j=1, \cdots, n$, we have
\begin{align*}
\frac{1}{2} E^j &= \frac{1}{2} \left ( A(x, t, \pa v, \sigma v) \frac{ \pa B}{\pa \xi_j}(x,t) - 
\frac{\pa A}{ \pa \xi_j}(x,t,\pa v, \sigma v) ( B(x,t, \pa v) + g(x,t) v)  \right )
\\
& = \phi_j ( |v_x|^2 - v_t^2) -  \sigma^2 \phi_j  (|\phi_x|^2 - \phi_t^2) v^2
- 2 v_j (v_x \cdot \phi_x - v_t \phi_t) -  g(x,t) v_j v
\end{align*}
and (index $0$ corresponds to $t$)
\begin{align*}
\frac{1}{2} E^0 & = 
\frac{1}{2} \left ( A(x, t, \pa v, \sigma v) \frac{ \pa B}{\pa \tau} - 
\frac{\pa A}{ \pa \tau}(x,t,\pa v, \sigma v) ( B(x,t,\pa v) + g(x,t) v)  \right )
\\
& =  - \phi_t ( |v_x|^2 - v_t^2) +  \sigma^2 \phi_t  (|\phi_x|^2 - \phi_t^2) v^2
+ 2 v_t (v_x \cdot \phi_x - v_t \phi_t) +  g(x,t) v_t v.
\end{align*}

\vspace{5pt}

%%%%%%%%%%%

\noindent
{\bf The boundary integrands on $\{ t=z \}$ when $\Omega = (B \times \R) \cap \{ t >z \}$.}

\vspace{10pt}

\noindent Here $x=(y,z)$ with $y \in \R^{n-1}$ and $\Omega = (B \times \R) \cap \{ t >z \}$ where
$B$ is the unit ball in $\R^n$. We compute the boundary integrand coming from $t=z$.
The outward normal to the part of $\pa \Omega$ on $t=z$ is $\sqrt{2} \nu = (\nu^y = 0, \nu^z = 1, \nu^t = -1)$.
Hence
\begin{align*}
\frac{1}{\sqrt{2}} \nu^j E^j  & = (\phi_z + \phi_t)(|v_x|^2 - v_t^2) 
- \sigma^2 (\phi_z + \phi_t) ( |\phi_x|^2 - \phi_t^2) v^2
\\
& \qquad - 2(v_z + v_t)(v_x \cdot \phi_x - v_t \phi_t)  - (v_z + v_t) g(x) v
\\
& =  (v_z + v_t) \left ( (\phi_z + \phi_t)(v_z - v_t) - 2 (v_z \phi_z - v_t \phi_t) \right )
+ (\phi_z + \phi_t)|v_y|^2  - 2(v_z + v_t)(v_y \cdot \phi_y)
\\
& \qquad - \sigma^2 (\phi_z + \phi_t) ( | \phi_x|^2 - \phi_t^2) v^2 - (v_z + v_t) g(x) v
\\
& =    (v_z + v_t) (- v_z \phi_z + v_t \phi_t + \phi_t v_z - \phi_z v_t) + (\phi_z + \phi_t)|v_y|^2
- 2(v_z + v_t)(v_y \cdot \phi_y)
\\
& \qquad - \sigma^2 (\phi_z + \phi_t) ( |\phi_x|^2 - \phi_t^2) v^2 - (v_z + v_t) g(x) v
\\
& =   (\phi_t - \phi_z) (v_z + v_t)^2 + (\phi_z + \phi_t)|v_y|^2 
- 2(v_z + v_t)(v_y \cdot \phi_y)
\\
& \qquad - \sigma^2 (\phi_z + \phi_t) ( |\phi_x|^2 - \phi_t^2) v^2 - (v_z + v_t) g(x) v.
%\\
%& =  \N \phi (\T v)^2 + (\T \phi) |v_y|^2 - 2 (\T v)( v_y \cdot \phi_y)  - \sigma^2 (\T \phi) ( - \T \phi \, \N \phi + |\phi_y|^2) v^2 - (\T v) g(x,t) v
\end{align*}
We adopt the notations 
\[
Z v := \frac{1}{\sqrt{2}}(v_z + v_t), \qquad N v := \frac{1}{\sqrt{2}}(v_t - v_z),
 \]
so that $Z$ is tangential and $N$ is normal to $t=z$. Thus the integrand in the boundary term over $t=z$ is given by 
\begin{align}
\nu^j E^j  &= 4 (N \phi) (Z v)^2 + 2 (Z \phi) |v_y|^2 - 4 (Z v)( v_y \cdot \phi_y)  \label{carleman_boundary_terms_gamma} \\
 & \qquad- 2 \sigma^2 (Z \phi) ( - 2 Z \phi \, N \phi + |\phi_y|^2) v^2 - 2 (Z v) g(x,t) v \notag \\
 &= 4 (N\phi) ( (Zv)^2 + \sigma^2 (Z\phi)^2 v^2) + 2 (Z\phi) ( \abs{v_y}^2 - \sigma^2 \abs{\phi_y}^2 v^2) \notag \\
 & \qquad  - 4 (Zv) (v_y \cdot \phi_y) - 2(Zv)gv. \notag
\end{align}

\bibliographystyle{alpha}

\end{document}